\documentclass[smallextended]{svjour3}     % onecolumn (second format)

\smartqed

\usepackage{graphicx}
\usepackage{fix-cm}
\usepackage
[colorlinks=false,urlcolor=blue,citecolor=blue,linkcolor=blue,bookmarks=true,
     bookmarksopen=false,pdftitle=created by dvipdf,pdfcreator=NM,
     pdfauthor=Megiddo,pdfsubject=mor.sty:6.29.04]{hyperref}

\usepackage{url}
\usepackage{subfigure}
\usepackage{amsfonts,mathrsfs}
\usepackage{amssymb,amsmath}
\usepackage{verbatim}
\usepackage{color}

\newcommand{\remove}[1]{}

\def\la{\langle}
\def\ra{\rangle}

\def\1{{\bf 1}}
\def\A{\mathscr{A}}

\def\A{\mathscr{A}}
\def\F{\mathscr{F}}

\def\ol{\bar}

\def\b{\beta}

\def\g{\gamma}
\def\a{\alpha}

\def\o{\omega}

\def\rn{\mathbb{R}^n}
\def\R{\mathbb{R}}

\def\re{\mathbb{R}}

\def\dist{{\rm dist}}

\newcommand{\EXP}[1]{\mathsf{E}\!\left[#1\right] }
\newcommand{\prob}[1]{\mathsf{Pr}\left\{ #1 \right\}}
\def\be{\begin{equation}}
\def\ee{\end{equation}}
\def\bit{\begin{itemize}}
\def\eit{\end{itemize}}
\def\argmin{\mathop{\rm argmin}}
\newtheorem{assumption}{Assumption}

\begin{document}

\title{Random minibatch subgradient algorithms for convex  problems with functional constraints}

\titlerunning{ Random minibatch subgradient algorithms }

\author{Angelia Nedi\'{c}   \and Ion Necoara}

\institute{
A.~Nedi\'{c} \at
School of Electrical, Computer and Energy Engineering\\
Arizona State University, Tempe, USA  \\
\email{angelia.nedich@asu.edu}.
\and
I. Necoara \at
Department of Automatic Control and Systems Engineering\\
University  Politehnica Bucharest, 060042 Bucharest, Romania\\
\email{ion.necoara@acse.pub.ro} (corresponding author).
}

\date{Received: 28 February 2019 / Accepted: date}

\maketitle

\begin{abstract}
In this paper we consider  non-smooth convex optimization problems with (possibly) infinite intersection of constraints. In contrast to the classical approach, where the constraints are usually represented as intersection of simple sets, which are easy to project onto, in this paper we consider that each constraint set is  given as the level set of a convex but not necessarily differentiable function. For these settings we propose  subgradient iterative algorithms with random minibatch feasibility updates. At each iteration, our algorithms take a subgradient step aimed at only minimizing the objective function and then a subsequent subgradient step  minimizing the feasibility violation of the observed minibatch of constraints. The feasibility updates are performed based on either  parallel or sequential random observations of several  constraint components.  We analyze the convergence behavior of the proposed algorithms for the case when the objective function is  strongly convex and with bounded subgradients, while the functional constraints are endowed with a bounded first-order black-box oracle. For a diminishing stepsize, we prove sublinear convergence rates   for the expected distances of the weighted averages of the iterates from the constraint set, as well as for the expected suboptimality of the function values along the weighted averages.  Our convergence  rates  are known to be optimal for subgradient  methods on  this class of problems. Moreover,  the rates depend explicitly on the minibatch size and show when minibatching helps a subgradient scheme with random feasibility updates.

\end{abstract}

%%%%%%%%%%%%%%%%%%%%%%%%%%%%%%%%%

\section{Introduction}
\label{sec:introduction}
\noindent The large sum of functions in the objective function and/or the large  number of constraints in  most of the practical optimization applications  led the stochastic optimization field to become an essential tool for many applied mathematics areas, such as machine learning and statistics \cite{MouBac:11,Vap:98}, constrained control  \cite{PatNec:17}, sensor networks  \cite{BlaHer:06}, computer science \cite{KunBac:18}, inverse problems \cite{BotHei:12},  operations research and finance \cite{RocUry:00}.  For example, in machine learning  applications the optimization algorithms involve numerical computation of parameters for a system designed to make decisions based on yet unseen data \cite{MouBac:11,Vap:98}. In particular, in support vector machines  one maps the data into a higher dimensional input space  and constructs an optimal separating  hyperplane  in  this  space by learning, eventually online, the hyperplanes corresponding to each data in the training set \cite{Vap:98}. This   leads to a convex optimization problem with a large number of functional constraints.

\vspace{0.2cm}

\noindent \textit{Contributions}. To deal with such  optimization  problems having (possibly) infinite number of functional constraints, we propose  subgradient  methods with random feasibility updates.  At each iteration, the algorithms take a subgradient step aimed at only minimizing the objective function, followed by a feasibility step  for  minimizing the feasibility violation of the observed minibatch of convex constraints achieved through the Polyak's subgradient iteration \cite{Pol:67,Pol:01}.  The feasibility updates in the  first algorithm are performed using  \textit{parallel}  random observations of several  constraint components, while in the second algorithm we consider   \textit{sequential} random observations of constraints. Both algorithms are reminiscent of a learning process where we try to learn the constraint set  while simultaneously minimizing an objective function.  The proposed algorithms are applicable to the situation where the whole constraint set of the problem is not known in advance, but it is rather learned in time through observations. Also,
these algorithms are of interest for (non-smooth) constrained optimization problems where the constraints are known but their number  is either large or not finite.

\vspace{0.1cm}

\noindent We study the convergence properties of the proposed random minibatch  subgradient   algorithms for the case  when the objective function need not be differentiable but it is  strongly convex, while the functional constraints are accessed trough a bounded first-order black-box oracle.  In doing so, we can avoid the need for projections to the set of  constraints, which may be  expensive computationally.  For a diminishing stepsize, we prove sublinear convergence rates of order ${\cal O}(1/t)$, where $t$ is the iteration counter,  for the expected distances of the weighted averages of the iterates from the constraint set, as well as for the expected suboptimality of the function values along the weighted averages.  Our convergence  rates  are known to be optimal for this class of subgradient schemes  for solving non-smooth convex problems with functional constraints.  Moreover, our rates  depend explicitly on the minibatch size  and show when minibatching works for a subgradient method with random feasibility updates.  \textit{To the best of our knowledge, this is the first work proving that subgradient methods with random minibatch  feasibility steps are better than their non-minibatch  variants}.  More explicitly,  the convergence estimate for the  parallel algorithm depends on a key parameter $L_N$ (see eq. \eqref{eq:Lnk} below), which determines whether  minibatching helps ($L_N <1$) or not ($L_N=1$) and how much (the smaller $L_N$, the better is the complexity), see Theorem 2. For the sequential variant, we show that minibatching always helps and the complexity depends exponentially  on the minibatch size (see Theorem 3).

\vspace{0.2cm}

\noindent \textit{Related works}.   In spite of its wide applicability, the study on efficient solution methods for optimization problems with many constraints  is still limited.  The most prominent work is the stochastic gradient descent (SGD) \cite{MouBac:11,NemJud:09,PolJud:92}. Even though SGD is a well-developed  methodology, it only applies to optimization problems with simple constraints, requiring the whole feasible set  to be  projectable.  A line of work that is known as alternating projections, focuses  on applying random projections for solving problems that are involving the intersection of a (infinite) number of sets. The case when the objective function is not present in the formulation,  which corresponds to the convex feasibility problem,  is studied e.g. in~\cite{BauBor:96,KeyZho:16,Ned:10,NecRic:18}. For this particular setting,  \cite{Ned:10,NecRic:18} combines the smoothing technique with (minibatch) SGD, leading to  stochastic alternating projection algorithms having linear convergence rates. In~\cite{PatNec:17} stochastic proximal point type steps are combined  with alternating projections for solving stochastic optimization problems with infinite intersection of sets.  Stochastic forward-backward algorithms have been also applied to solve optimization problems with many constraints. However, the papers introducing those general algorithms focus on proving only assymptotic convergence results and do not derive convergence rates, or they assume the number of constraints is finite, which is more restricted than our settings ~\cite{BiaHac:17,SheTeb:14,WanChe:15}. In the case where the number of constraints  is finite and the objective function is deterministic, Nesterov's smoothing framework is studied in~\cite{BotHen:13,OuyGra:12,TraFer:18} in the setting of accelerated proximal gradient methods.   Incremental subgradient  methods or primal-dual approaches were also proposed for solving  convex optimization problems with finite intersection of simple sets through an exact penalty reformulation  in ~\cite{Ber:11,KunBac:18}. 

\vspace{0.1cm}

\noindent The paper most related to our work is \cite{Ned:11}, see also \cite{Pol:67,Pol:69,Pol:01},  where iterative  subgradient methods with random feasibility steps are  proposed for solving convex problems with functional constraints.  Our algorithms are minibatch extensions of the algorithm proposed in~\cite{Ned:11}.   Moreover,  in~\cite{Ned:11} only sublinear convergence rates of order ${\cal O}(1/\sqrt{t})$ are established  for convex objective functions,  while in this paper  we show that ${\cal O}(1/t)$ rates are valid under a relaxed strong convexity condition.  Finally, since we deal with minibatching and  a relaxed strong convexity assumption, our convergence analysis  requires additional insights that differ from that of~\cite{Ned:11}.   Similarly, in~\cite{PatNec:17} a  stochastic optimization problem with infinite intersection of sets is considered and  stochastic proximal point steps are combined  with alternating projections for solving it.  However,  in order to prove sublinear convergence rates  ${\cal O}(1/t)$, \cite{PatNec:17} requires   strongly convex and smooth objective functions,  while our results are valid for a more relaxed  strong convexity condition and possible non-smooth fuctions. Lastly,~\cite{PatNec:17} assumes the projectability of individual sets, whereas in our case, the constraints might not be projectable.

\vspace{0.2cm}

\noindent \textit{Notation}. The inner product of two vectors $x$ and $y$  in  $\R^n$ is denoted by $\la x,y\ra$,  while $\|x\|$ denotes the standard Euclidean norm. We write $\dist(\ol x, X)$ for the distance of a vector $\ol x$  from a closed convex set $X$, i.e., $\dist(\ol x,X) =\min_{x\in X}\|x-\ol x\|$, while $\Pi_X[\ol x]$ denotes the projection of  $\ol x$ onto $X$, i.e., $\Pi_X[\ol x]=\argmin_{x\in X}\|x-\ol x\|^2$. For a scalar $a$, we write  $a^+=\max\{a, 0\}$. For a convex function $h$, we denote $s_h(x)$  a subgradient of $h$ at  $x$ and $\partial h(x)$ denote the set of all subgradients of $h$ at $x$. If $h$ is differentiable at $x$, then its gradient is denoted $\nabla h(x)$.  We write $\prob{\o}$ and $\EXP{\o}$ to denote respectively  the probability distribution and the expectation of a random variable~$\o$.   Finally,  the big $\mathcal{O}$ notation, i.e.  $f(t) \leq \mathcal{O}(g(t))$,  means that there exist $C>0 \; \text{and} \; t_0$  \text{such that}  $f(t) \leq C \cdot g(t)$ for all $t \geq t_0$.   

\vspace{0.2cm}

\noindent  \textit{Outline}. The content of the paper is as follows. In Section  \ref{sec:problem} we introduce our problem of interest and the main assumptions. In Section  \ref{sec:parallel} we propose a parallel random minibatch subgradient algorithm and derive its convergence rate, while in Section \ref{sec:convergence_diff_seq} the sequential variant is analyzed.  Finally, in Section \ref{sec:extensions} we discuss some extensions, while in Section \ref{sec:sim} we report some preliminary numerical results.

%%%%%%%%%%%%%%%%%%%%%%%%%%%%%%%%%%

\subsection{Problem formulation}
\label{sec:problem}
In this paper we are interested in solving the following convex constrained minimization problem:
\begin{eqnarray}
\label{eq-problem}
\hbox{minimize\ \,}&& f(x)\cr
\hbox{subject to} && x\in X,
\quad X\triangleq Y\cap\left(\cap_{\o \in \A} X_\o)\right),\\
\hbox{with \,}\qquad
&& X_\o=\{x\in\re^n\mid g_\o(x) \le 0\} \qquad \hbox{for every $ \o \in \A$}, \nonumber
\end{eqnarray}
where $\A$ is an arbitrary collection of indices and $Y$ is a closed convex set. The objective function $f$ and all constraint functions $g_\o$ are assumed convex. We also assume that the optimization problem \eqref{eq-problem} has finite optimum and we let  $f^*$ and $X^*$ denote the optimal value and the optimal set, respectively,
 \[f^*=\inf_{x  \in X}f(x),\qquad X^*=\{x\in X\mid f(x)=f^*\}.\]
We work under the premise that the collection $\A$ is large, possibly infinite (even uncountable).   Such problems have many applications in engineering, machine learning, computer science,  operations research and  finance  \cite{MouBac:11,Vap:98,PatNec:17,BotHei:12,RocUry:00}.  Let us now  formally state the assumptions on the functions $f$ and $g_\o$,  with $\o \in\A$, of problem \eqref{eq-problem}.

%One motivating example is stochastic programming  with the conditional value at risk  constraint. In an important work \cite{RocUry:00}, Rockafellar and Uryasev shows that a class of asset allocation problem can be modeled as:
%\begin{eqnarray*}
%&& \hbox{minimize}_{x,\tau}  \;\; c^T x \cr
%&& \hbox{subject to} \quad  x \geq 0, \;\; \sum_{i=1}^n x_i=1, \;\;   \delta \tau +(-\o^T x - \tau)^+  \leq 0 \quad \forall \o \in \A,
%\end{eqnarray*}
%where $\o$  denotes the random return with mean $c = \EXP{\o}$. Other practical  examples are given in Section \ref{sec:sim}. }  

\begin{assumption}\label{asum-base}
Let the following hold:
\begin{itemize}
\item[(a)] The set $Y$ is closed, convex and simple (i.e., easy for projection). The constraint set $X$ and the optimal set $X^*$ are non-empty. 
\item[(b)] The objective function $f: \re^n \to \bar{\re}$ is  strongly convex on the set $Y$  with a  constant $\mu>0$, i.e.:
\[ f(y)  \ge f(x)  + \la s_f(x), y - x \ra  +  \frac{\mu}{2} \|y - x\|^2 \quad \forall x, y \in Y, \; s_f(x) \in \partial f(x).  \]
The subgradients of the function $f$ are uniformly bounded on the set $Y$, i.e.,  there is $M_f>0$ such that
\[  \|s_f(x)\|\le M_f \quad \forall  s_f(x) \in \partial f(x) \; \hbox{and} \;  x \in Y. \]
\item[(c)]  The functional constraints $g_\o: \re^n \to \bar{\re}$ are convex, not necessarily differentiable, and have bounded subgradients on the set $Y$, i.e., there is $M_g>0$ such that
\[ \|d\| \le M_g \quad  \forall d \in \partial g_\o(x),  \; x \in Y \; \hbox{and } \;  \o \in \A.  \]
\end{itemize}
\end{assumption}

\noindent  We assume,  that the domains of definition of the functions $f$ and $g_\o$ contain $Y$.    It follows immediately from Assumption~\ref{asum-base}(b) that (see e.g.,  \cite{NecNes:15}):
\begin{align*}
%\label{eq:rsc}
f(x) - f^* \geq \frac{\mu}{2} \|x - x^*\|^2  \quad  \forall x \in X \subseteq Y.
\end{align*} 
Note that the conditions of Assumption~\ref{asum-base}(b) may look contradictory since the following relations need to hold:
\[   \frac{\mu}{2} \|  x - x^*\|^2 \leq f(x) - f^* \leq \la s_f(x), x - x^* \ra \leq M_f \| x - x^*\| \quad \forall x \in X, x^* \!\in\! X^*\!,  \]
where the second inequality follows from the convexity of $f$ and the third one from the Cauchy-Schwartz inequality.  This implies that $\|x - x^*\| \leq 2 M_f/\mu$ for any  $ x \in X  \subseteq  Y$. Note that this inequality is always valid provided that the set  $Y$ is compact and our optimization model \eqref{eq-problem} allows us to impose such an assumption on the set $Y$.  Moreover,  when the sets $X_\o$ are simple for projection  operation, then one may choose an alternative equivalent description of the constraint sets by letting $g_\o(x) = \dist (x,X_\o)$ for all $x \in \re^n$.  Note that in this case  $d(x) = \frac{x - \Pi_{X_\o}[x]}{\dist (x,X_\o)}  \in \partial g_\o(x)$ for all $x \not \in X_\o$. Furthermore, $\|d(x)\| =1$, thus the subgradients are bounded with $M_g=1$ in this case. Therefore, our approach is more general than those from most of the existing works,  which usually assume projectability of each $X_\o$  (see also \textit{Related works} paragraph from Section \ref{sec:introduction}).

%%%%%%%%%%%%%%%%%%%%%%%%%%%%%%
%%%%%%%%%%%%%%%%%%%%%%%%%%%%%%

\section{Parallel random minibatch subgradient  algorithm}
\label{sec:parallel}
To solve the convex problem with functional constraints~\eqref{eq-problem}, we first propose a subgradient method with \textit{parallel} random minibatch feasibility updates. More precisely,  our first algorithm is a parallel minibatch extension of the algorithm  proposed in~\cite{Ned:11}, leading to the following iterative process:
\begin{center}
\framebox{
\parbox{11 cm}{
\begin{center}
\textbf{ Algorithm (parallel case) }
\end{center}
Choose $x^0 \in Y$, minibatch size $N \ge 1$, and 
stepsizes $\alpha_k >0$ and $\b > 0$. For $k \geq 1$ repeat:
\begin{enumerate}
\item Draw $N$  samples $J_k=\{\o^k_1, \cdots, \o^k_N\} \sim \textbf{P}$.
\item Compute the following updates: 
\begin{subequations}
\label{eq-method}
\begin{eqnarray}
&& v_k = \Pi_Y[x_{k-1} - \a_{k-1} s_f(x_{k-1})],
\label{eq-grad-update}\\
&& z_{k}^i= v_{k}- \beta\, \frac{g^+_{\o_{k}^{i}}(v_{k})}{\|d_{k}^i\|^2}\, d_{k}^{i} 
\quad \hbox{for } i=1:N,
\label{eq-batch}\\
&& x_k=\Pi_Y[ \bar z_k], \quad \hbox{with } \bar z_k =\frac{1}{N}\sum_{i=1}^N z_{k}^{i}.\label{eq-feas-update}
\end{eqnarray}
\end{subequations}
\end{enumerate}
}}
\end{center}

\noindent Here,  $\a_k>0$ and $\beta>0$ are deterministic stepsizes and recall that $s_f(x)$ denotes a subgradient of $f$ at $x$ and $g_\o^+(x) = \max \{g_\o(x), 0\}$.  The method takes one subgradient step for the objective function, followed by $N$ feasibility updates in \textit{parallel}, which are then averaged  and projected onto the set $Y$.   In a parallel implementation, we assume  available $N+1$  cores collocated on the same machine, of which one is designated as a central core; the central core sends $v_k$  to all other cores, which perform the update \eqref{eq-batch}  and  send their updates to the central core; finally the central core  performs the average step \eqref{eq-feas-update} and the optimality step \eqref{eq-grad-update}.  We note that at each of the feasibility update step $N$ random constraints are selected from the collection of the constraint sets according to the probability distribution $\textbf{P}$, i.e., the index variable $\o_{k}^{i}$ is random with values in the set $\A$. The vector $d_{k}^{i}$ is chosen as $d_{k}^{i} \in \partial g^+_{\o_{k}^{i}}(v_{k})$ if $g^+_{\o_{k}^{i}}(v_{k})>0$ and $d_{k}^{i}=d$
 for some $d \ne 0$ if $g^+_{\o_{k}^{i}}(v_{k})=0$. When  $g^+_{\o_{k}^{i}}(v_{k})=0 $, we have $z_{k}^{i}=v_k$ for any choice of $d\ne0$.   Note that the feasibility step \eqref{eq-batch}  has the special form of Polyak’s subgradient iteration, see e.g.,  \cite{Pol:67,Pol:01}.   Moreover, when $X_\o$ are projectable, then one chooses $g_\o(x) = g_\o^+(x) =  \dist (x,X_\o)$ for all $x \in \re^n$ and the update  \eqref{eq-batch} becomes a usual projection step:
\[ z_k^i = v_k - \beta(v_k  - \Pi_{X_{\o_{k}^{i}}}[v_k]).   \]  
\noindent The initial point $x_0\in Y$ is selected randomly with an arbitrary distribution. The projection on the set $Y$ in the updates \eqref{eq-grad-update} and \eqref{eq-feas-update}  is used to ensure that each $v_k$ and  $x_k$ remain in the set $Y$,  over which the functions $f$ and $g_\o$ are assumed to have bounded subgradients.   Our next assumption deals with the random variables $\o_k^i$ for $i=1:N$ chosen according to the probability distribution $\textbf{P}$. For this, we  introduce the sigma-field $\F_{k}$ induced by the history of the method, i.e.,  by the realizations of the initial point $x_0$ and the variables  $\o_t^i$ up to main iteration~$k$:
\[ \F_{k} = \{x_0\} \cup  \left\{\o_{t}^{j} \mid \;  1\le t\le k, \; 1\le j\le N \right\}, \]
which contains the same information as the set $\{x_0\}\cup \{\{v_{t}, x_t\}\mid 1\le t\le k\}$. For notational convenience, we will allow $k=0$ by letting $\F_{0}=\{x_0\}$. We impose the following assumption.

\begin{assumption}
\label{asum-regularmod}
There exists a constant $c \in (0, \infty)$ such that
\[\dist^2(y,X)\le c \cdot \EXP{(g^+_{\o_{k}^i}(y))^2\mid \F_{k-1}} \quad \forall  y\in Y, \; k \ge1 \; \hbox{and } i=1,\ldots, N.  \]
\end{assumption}
Assumption~\ref{asum-regularmod} does not require that $J_k=\{\o_k^1,\ldots,\o_k^N\}$ are conditionally independent, given $\F_{k-1}$. For example, when the collection $\A$ is finite, the indices $i\in \A$ can be selected randomly without  replacement, i.e., given the realizations of  $\o_k^1=j_1,\ldots,\o_k^{i-1}=j_{i-1}$, the index $\o_k^i$ can be random with realizations in  $\A\setminus\{j_1,\ldots,j_{i-1}\}$. As another example, the index set $\A$ can be partitioned in $N$ disjoint sets  $\cup_{i=1}^N \A_1=\A$, and each $w_k^i$ can be uniformly distributed over the index set $\A_i$. Such a sampling allows for a parallel computation of all $z_k^i$ in the algorithm~\eqref{eq-method}. One can also combine the preceding two possibilities, by using a smaller partition of the set $\A$, and in each of the partitions choose the corresponding $\o_k^i$ sequentially, without replacement.  Assumption~\ref{asum-regularmod} is crucial in our  convergence analysis of method~\eqref{eq-method}. It summarizes all the information we need regarding the  distributions of the random variables $\o_{k}^i$ and the initial point $x_0$.   A discussion on the equivalence between  the  Assumption~\ref{asum-regularmod} and the linear regularity condition for the sets $(X_\o)_{\o \in \A}$ can be found in~\cite{Ned:10,Ned:11,NecRic:18}.  When each set $X_\o$ is given by  a linear inequality $a_\o^T x + b_\o \leq 0$, one can verify that  the  intersection of these halfspaces  over any arbitrary  index set  $\A$ is linearly regular provided that the sequence $(a_\o)_{\o \in \A}$ is bounded, see  \cite{BurFer:93,FerNec:19}. Hence,  Assumption~\ref{asum-regularmod} is also satisfied in this case. Moreover,  Assumption ~\ref{asum-regularmod}  holds  provided that the interior of the intersection over the arbitrary  index set  $\A$  has an interior point \cite{Pol:01}.  However, Assumption~\ref{asum-regularmod} holds for more general sets, e.g., when a strengthened Slater condition holds for a  collection of convex functional constraints $(X_\o)_{\o \in \A}$, such as the generalized Robinson condition, as detailed in Corollary 2 of~\cite{LewPan:98}.

%%%%%%%%%%%%%%%%%%%%%%%%%%%%

\subsection{Preliminary results}
\label{sec:preliminary}
In this section, we  derive some  preliminary results for later use in the convergence analysis of method \eqref{eq-method}. We start by recalling  a  basic property of the projection operation  on a closed convex set $Y\subseteq\rn$ \cite{Ned:10}:
\begin{equation}
\label{eqn:projection}
\|\Pi_Y[v] - y \|^2 \le \|v-y\|^2 - \|\Pi_Y[v]-v\|^2 \qquad\hbox{for any $v\in\rn$ and $y\in Y$.}
\end{equation}

\noindent We now show that the parameter $c$ in Assumption~\ref{asum-regularmod}   satisfies the following inequality:
\begin{lemma}
\label{lema:c}
Let Assumption~\ref{asum-base}(c) and Assumption~\ref{asum-regularmod} hold. Then, we have:
\[  c M_g^2 \geq 1. \]
\end{lemma}

\begin{proof}
Let $y \in Y$ be such that $y \not \in X$. Then, there exists $\bar{\o} \in \A$ such that the convex function $g_{\bar{\o}}$ satisfies $g_{\bar{\o}}(y) >0$.  Consequently, for any $s_g(y) \in \partial g_{\bar{\o}}(y)$  we also have $s_g(y) \in \partial g_{\bar{\o}}^+(y)$, and using convexity of $g_{\bar{\o}}^+$, we obtain:
\begin{align*}
0 = g_{\bar{\o}}^+(\Pi_X[y]) \geq g_{\bar{\o}}^+(y) + \la s_g(y), \Pi_X[y] - y \ra \geq g_{\bar{\o}}^+(y) - M_g \|\Pi_X[y] -y\|,
\end{align*}
or equivalently
\[  g_{\bar{\o}}^+(y)  \leq  M_g \|\Pi_X[y] -y\|.  \]
On the other hand for those $\o \in \A$ for which $g_\o(y) =0$ we automatically have
\[  0 = g_\o^+(y)  \leq  M_g \|\Pi_X[y] -y\|.  \]
In conclusion, for any $\o \in \A$ there holds:
\[  g_\o^+(y)  \leq  M_g \|\Pi_X[y] -y\|.  \]
Combining the preceding  inequality and Assumption~\ref{asum-regularmod}, we obtain:
\begin{align*}
\dist^2(y,X) & = \|\Pi_X[y] -y\| ^2 \leq c  \EXP{(g^+_{\o_{k}^i}(y))^2\mid \F_{k-1}} \\
& \leq   c  \EXP{ M_g^2  \|\Pi_X[y] -y\| ^2  \mid \F_{k-1}} = c M_g^2 \ \dist^2(y,X),
\end{align*}
which proves our relation $c M_g^2 \geq 1$.  \hfill$\square$
\end{proof}

\noindent We now derive a relation between the iterates $v_{k+1}$ and $x_{k}$.
\begin{lemma}
\label{lem-viter}
Let Assumptions~\ref{asum-base}(a) and~\ref{asum-base}(b) hold. Let $v_{k+1}$ be obtained via equation~\eqref{eq-grad-update} for a given $x_k\in Y$.  Then,  for the unique  optimal solution $x^*$ of the problem \eqref{eq-problem} and  $\rho \in (0, \; 1)$, we have:
\begin{align*}
& \|v_{k+1} - x^*\|^2  +  2\a_k(1-\rho)\left(f(\Pi_X[x_k])- f^*\right) \\
&\qquad \qquad   \le (1 - \a_k\rho\mu)\|x_k - x^*\|^2 + 2\a_k M_f\|\Pi_X[x_k] - x_k\| +\a_k^2M_f^2.
\end{align*}
\end{lemma}

\begin{proof}
Using the standard analysis of the projected subgradient method and the fact that the subgradients    of $f$ are uniformly bounded on $Y$, we have for the  optimal solution $x^*$ of  \eqref{eq-problem}  the following inequality, see e.g.,  \cite{Pol:67,Pol:69}:
\be
\label{eq-vkrel}
\|v_{k+1} -  x^*\|^2 \le \|x_k - x^*\|^2 - 2\a_k \left(f(x_k)-f(x^*)\right) +\a_k^2 M_f^2.
\ee
We provide a lower bound on $f(x_{k}) - f(x^*)$.  We consider two choices, namely, one is based on the  strong convexity of $f$ and the other is based on  considering another intermittent point.  By the strong convexity of $f$, we have
\begin{align}
\label{eq-split1}
& f(x_{k}) - f(x^*)  \ge \la s_f(x^*), x_k - x^* \ra +   \frac{\mu}{2}\|x_k - x^*\|^2 \nonumber \\
&  = \la s_f(x^*), \Pi_X[x_k] - x^* \ra +   \la s_f(x^*), x_k - \Pi_X[x_k] \ra  +  \frac{\mu}{2}\|x_k - x^*\|^2 \nonumber \\
& \geq  \la s_f(x^*), x_k - \Pi_X[x_k] \ra  +  \frac{\mu}{2}\|x_k - x^*\|^2 \nonumber \\
& \geq - M_f \|\Pi_X[x_k] - x_k\|  +  \frac{\mu}{2}\|x_k - x^*\|^2,
\end{align}
where the second inequality follows from the optimality conditions for $x^*$ and the last inequality follows from the Cauchy-Schwartz and  boundedness of the subgradients of $f$ on $Y$. The other choice consists of adding and subtracting $f(\Pi_X[x_k])$, which yields
\begin{align*}
f(x_{k}) - f(x^*) & = f(x_{k}) - f(\Pi_X[x_k]) + f(\Pi_X[x_k]) - f(x^*) \cr
& \ge-\|s_f(\Pi_X[x_k])\| \, \|\Pi_X[x_k]-x_k\| + f(\Pi_X[x_k])- f(x^*),
\end{align*}
where the last inequality follows by the convexity of $f$ and the Cauchy-Schwarz inequality. By Assumption~\ref{asum-base}(b), the subgradients of $f$ are uniformly bounded on $Y$ and hence, also     on $X$, implying that
\be
\label{eq-split2}
f(x_{k}) - f(x^*) \ge f(\Pi_X[x_k]) - f(x^*)  - M_f \|\Pi_X[x_k] - x_k\|.
\ee
We now let $\rho \in (0, 1)$ be arbitrary. By multiplying relation~\eqref{eq-split1} with $\rho$ and relation~\eqref{eq-split2} with $(1-\rho)$, and by adding the resulting relations, we obtain
\begin{eqnarray}
\label{eq-fvfy}
f(x_{k}) - f(x^*) & \ge &  \frac{\rho \mu}{2} \|x_k - x^*\|^2 \\
&& +(1-\rho) \left(f(\Pi_X[x_k]) - f(x^*)\right) -  M_f \|\Pi_X[x_k]-x_k\|. \nonumber
\end{eqnarray}
By using the estimate~\eqref{eq-fvfy} in relation~\eqref{eq-vkrel}, we obtain
\begin{eqnarray}\label{eq-estvkxk0}
\|v_{k+1}-x^*\|^2
     &\le& (1- \a_k\rho\mu)\|x_k - x^*\|^2  -2 \a_k (1-\rho) \left(f(\Pi_X[x_k]) - f(x^*)\right) \cr
     &&+ 2\a_k M_f\|\Pi_X[x_k]-x_k\| +\a^2_k M_f^2,
\end{eqnarray}
and after  re-arranging some of the terms we get the relation of the lemma.  \hfill$\square$
\end{proof}

\begin{remark}
The best choice for the parameter $\rho$ is not apparent at this point. It is important to have it in order to have the function value involved in the expression,  but it can be that $\rho=\frac{1}{2}$  will just do fine.
\end{remark}

% OLD LEMMA
%\begin{lemma}
%\label{lemma:note}
%Consider a convex function $g: \rn \to \re$ and a given point $x$. Let $d \in \partial g(x)$, with $d \ne 0$, and $\L_g(0)$ be the level set of the function $g$ associated with the value 0, i.e., $\L_g(0)=\{z \in \re^n \mid g(z) \le 0\}$. Then,
%\[ \frac{g^+(x)}{\|d\|} \le \dist(x, \L_g(0)). \]
%\end{lemma}
%
%\begin{proof}
%The relation is trivially satisfied if $x \in \L_g(0)$, so let $x$ be such that $g(x)>0$. Hence, by the convexity of $g$, it follows that for any $y \in \rn$,    \[g^+(x)=g(x) \le g(y) +\la s_g(x),x-y\ra, \qquad \hbox{for any } s_g(x) \in \partial g(x). \]
%Letting $s_g(x)=d$ and $y$ be the projection of $x$ on the level set $\L_g(0)$, we get
%\[ g^+(x)  \!\le\!  g(\Pi_{\L_g(0)}[x])  + \la d,x - \Pi_{\L_g(0)}[x] \ra \!\le\! \la d, x - \Pi_{\L_g(0)}[x]\ra \!\le\! \|d\|  \|x - \Pi_{\L_g(0)}[x]\|.\] Since $d\ne0$, the desired relation follows. \hfill$\square$
%\end{proof}

\noindent We next state a result that will be used to provide a basic relation between the iterates $v_k$ and $x_{k-1}$. The relation is stated in a generic form, and its proof can be found in \cite{Pol:69,Pol:67}.
\begin{lemma} \cite{Pol:69,Pol:67}
\label{lemma:basiter}
Let $g$ be a convex function over a closed convex  set $Z$, and let $y$ be given by
 \[ y = \Pi_{Z} \left[ v - \b  \frac{g^+(v)}{\|d\|^2}\, d\right]  \qquad \hbox{for $v \in Z$, $d \in\partial g^+(v)$ and $\beta>0$},\]
where $d\ne0$. Then, for any $\bar z \in Z$ such that $g^+(\bar z)=0$, we have
\[ \|y -\bar z\|^2 \le  \|v -\bar z\|^2 - \b(2-\beta)\,\frac{(g^+(v))^2}{\|d\|^2}.\]
\end{lemma}

\noindent In the analysis, we will also make use of the relation for averages,  stating that for  given vectors $u_1,\ldots, u_N\in\rn$ and their average $\bar u=\frac{1}{N}\sum_{i=1}^N u_i$, the following relation is valid for any vector $w\in\re^n$:
\be\label{eq-averages}
\|\bar u-w\|^2=\frac{1}{N}\sum_{i=1}^N\|u_i-w\|^2 -\frac{1}{N}\sum_{j=1}^N\|u_j - \bar u\|^2.\ee

\noindent Now we provide a basic relation for the iterate $x_k$ upon completion of the $N$ randomly sampled feasibility updates.
\begin{lemma}\label{lem-xiter}
Let Assumption~\ref{asum-base}(a)  hold. Let $x_{k}$ be obtained via updates~\eqref{eq-batch}  and~\eqref{eq-feas-update} for a given $v_k\in Y$ and $\beta >0$.  Then, the following relation holds:
% OLD DERIVATIONS
%\begin{align*}
% & \|x_{k}-y\|^2\le \|v_{k} - y\|^2  - \frac{\b(2-\b)}{N \, M_g^2}\sum_{i=1}^N(g_{\o_{k}%^{i}}^+(v_k))^2  -\frac{1}{N}\sum_{i=1}^N\|z_k^i - x_k \|^2 \; \forall y\in  X, \\
%& \dist^2(x_{k}, X)\le \dist^2(v_{k},X)  - \frac{\b(2-\b)}{N \, M_g^2} \sum_{i=1}^N(g_{\o_{k}^{i}}^+(v_k))^2  -\frac{1}{N}%\sum_{i=1}^N\|z_k^i - x_k\|^2.
%\end{align*}
%
\begin{align*}
& \dist^2(x_{k}, X)\le \dist^2(v_{k},X)  - \frac{\b(2-\b)}{N}\sum_{i=1}^N\frac{(g_{\o_{k}^{i}}^+(v_k))^2}{\|d_k^i\|^2}   - \b^2 V_N(v_k),
\end{align*}
where $V_N(v_k)$ is the total variation of the minibatch subgradients, i.e.,
\[V_N(v_k)=\frac{1}{N}\sum_{i=1}^N\left\|\frac{g_{\o_{k}^{i}}^+(v_k)}{\|d_k^i \|^2}\, d_k^i - \frac{1}{N}\sum_{j=1}^N\frac{g_{\o_{k}^{j}}^+(v_k)}{\|d_k^j\|^2}\, d_k^j\right\|^2.\]
\end{lemma}

\begin{proof}
By the projection property  \eqref{eqn:projection} and the definition of $x_k$, we have for any $y\in  X$ that:
\be
\label{eq-averxk0}
\|x_k-y\|^2\le  \left\|\bar z_k - y\right\|^2 - \left\|x_k - \bar z_k\right\|^2.
\ee
By the definition we have $\bar z_k=\frac{1}{N}\sum_{i=1}^N z_{k}^{i}$. Thus,  by using relation~\eqref{eq-averages} for the collection $z_k^1,\ldots,z_k^N$, we have for any $w\in \rn$,
\be\label{eq-zk-aver} \left\| \bar z_k - w\right\|^2=\frac{1}{N}\sum_{i=1}^N\|z_k^i-w\|^2
       -\frac{1}{N}\sum_{i=1}^N \left\| z_k^i - \bar z_k \right\|^2.\ee
Letting $w=y$ in the preceding relation and combining the resulting relation with~\eqref{eq-averxk0}, we obtain
\[\|x_k-y\|^2\le \frac{1}{N}\sum_{i=1}^N\|z_k^i - y\|^2
       -\frac{1}{N}\sum_{i=1}^N \left\| z_k^i - \bar z_k\right\|^2  - \|x_k - \bar z_k\|^2.\]
Now, we use the definition of the iterates $z_{k}^i$ in algorithm~\eqref{eq-method}  and Lemma~\ref{lemma:basiter}, with $Z=\re^n$. Thus, we obtain  for any $y\in X$ (for which we would have $g_{\o_{k}^i}^+(y)=0$  for any realization of $\o_{k}^i$) and for any $i=1,\ldots,N$,
 \[\|z_{k}^i - y\|^2 \le  \|v_{k} - y\|^2 -\b(2-\beta)\,\frac{(g_{\o_{k}^{i}}^+(v_k))^2}{\|d_k^i\|^2}.\]
 Hence, it follows that for any $y \in X$,
       \[\|x_{k}-y\|^2\le \|v_{k} - y\|^2  - \frac{\b(2-\b)}{N}\sum_{i=1}^N\frac{(g_{\o_{k}^{i}}^+(v_k))^2}{\|d_k^i\|^2}
       -\frac{1}{N}\sum_{i=1}^N\|z_k^i -\bar z_k\|^2 - \|x_k - \bar z_k\|^2.\]

% OLD DERIVATIONS
%In view of~\eqref{eq-zk-aver}, with $w=x_k$, we have
%       \[\frac{1}{N}\sum_{i=1}^N\|z_k^i -\bar z_k\|^2 + \|x_k - \bar z_k\|^2=
 %      \frac{1}{N}\sum_{i=1}^N\|z_k^i -x_k \|^2.\]

\noindent From the definition of the iterates $z_{k}^i$ in algorithm~\eqref{eq-method}, we see that
\[\|z_k^i -\bar z_k\|^2 =\b^2\left\|\frac{g_{\o_{k}^{i}}^+(v_k)}{\|d_k^i \|^2}\, d_k^i
- \frac{1}{N}\sum_{j=1}^N\frac{g_{\o_{k}^{j}}^+(v_k)}{\|d_k^j\|^2}\, d_k^j\right\|^2.\]
By defining
\[V_N(v_k)=\frac{1}{N}\sum_{i=1}^N \left\|\frac{g_{\o_{k}^{i}}^+(v_k)}{\|d_k^i \|^2}\, d_k^i
- \frac{1}{N}\sum_{j=1}^N\frac{g_{\o_{k}^{j}}^+(v_k)}{\|d_k^j\|^2}\, d_k^j\right\|^2,\]
we have
\[\frac{1}{N}\sum_{i=1}^N\|z_k^i -\bar z_k\|^2=\b^2 V_N(v_k).\]
Therefore, we obtain for any $y \in X$,
\begin{align}
\label{eq:lema4y}
\|x_{k}-y\|^2\le \|v_{k} - y\|^2  - \frac{\b(2-\b)}{N} \sum_{i=1}^N\frac{(g_{\o_{k}^{i}}^+(v_k))^2}{\|d_k^i\|^2} -  \b^2 V_N(v_k).
\end{align}
The statement of the lemma  follows by letting $y=\Pi_X[v_k]$ in the preceding relation
and using the  fact that $\|x_k - \Pi_X[x_k]\|\le  \|x_{k} - \Pi_X[v_k\|\|$.   \hfill$\square$
\end{proof}

\noindent Let us define the following  parameters:
\begin{align}
\label{eq:Lnk}
L_N^k =   \left\| \frac{1}{N}\sum_{i=1}^N\frac{g_{\o_{k}^{i}}^+(v_k)}{\|d_k^i\|^2}\, d_k^i\right\|^2 \Big{/} \frac{1}{N} \sum_{i=1}^N\frac{(g_{\o_{k}^{i}}^+(v_k))^2}{\|d_k^i\|^2}   \quad \text{and}  \quad L_N = \max_{k \geq 0} L_N^k.
\end{align}
From Jensen's inequality it follows that $L_N^k \leq 1$. However, there are also convex functions $g_\o$ such that $L_N^k < 1$. We postpone the derivation of such examples of functional constraints satisfying condition $L_N^k < 1$ until  Section \ref{sec:Lnk}.  The parameter $L_N \leq 1$ will play a key role in our derivations below. In particular, we obtain the following simplification for   Lemma~\ref{lem-xiter}.

\begin{lemma}
\label{lem-xiter1}
Let Assumptions~\ref{asum-base}(a) and \ref{asum-base}(c)  hold. Let $L_N \leq 1$ as defined in \eqref{eq:Lnk} and $x_{k}$ be obtained via updates~\eqref{eq-batch}  and~\eqref{eq-feas-update} for a given $v_k\in Y$ and extrapolated stepsize $\beta \in (0, \, 2/L_N)$.  Then, the following relation holds:
\begin{align*}
& \dist^2(x_{k}, X) \le \dist^2(v_{k},X)  - \frac{\b(2 - \b L_N)}{N  M_g^2} \sum_{i=1}^N (g_{\o_{k}^{i}}^+(v_k))^2.
\end{align*}
\end{lemma}

\begin{proof}
Note that the total variation of the minibatch subgradients $V_N(v_k)$ can be written equivalently as:
\[  V_N(v_k)  = \frac{1}{N} \sum_{i=1}^N\frac{(g_{\o_{k}^{i}}^+(v_k))^2}{\|d_k^i\|^2} - \left\| \frac{1}{N}\sum_{i=1}^N\frac{g_{\o_{k}^{i}}^+(v_k)}{\|d_k^i\|^2}\, d_k^i\right\|^2.  \]
Using the previous expression of $V_N$ and the definitions of $L_N^k$ and $L_N$ from  \eqref{eq:Lnk} in Lemma ~\ref{lem-xiter}, we get:
\begin{align*}
& \dist^2(x_{k}, X) \le \dist^2(v_{k},X)  - \frac{\b(2-\b)}{N}\sum_{i=1}^N\frac{(g_{\o_{k}^{i}}^+(v_k))^2}{\|d_k^i\|^2}   - \b^2 V_N(v_k) \\
& = \dist^2(v_{k},X)  - \frac{\b(2-\b)}{N} \sum_{i=1}^N\frac{(g_{\o_{k}^{i}}^+(v_k))^2}{\|d_k^i\|^2}  - \frac{\b^2 (1 - L_N^k)}{N} \sum_{i=1}^N\frac{(g_{\o_{k}^{i}}^+(v_k))^2}{\|d_k^i\|^2}\\
&  \leq  \dist^2(v_{k},X)  - \frac{\b(2 - \b L_N)}{N} \sum_{i=1}^N\frac{(g_{\o_{k}^{i}}^+(v_k))^2}{\|d_k^i\|^2}.
\end{align*}
By Assumption~\ref{asum-base}(c)   each function $g_i$ has bounded subgradients uniformly on $Y$. Hence,  we have $\|d_k^i\| \le M_g$, which used  in the previous inequality implies the statement of the lemma.  \hfill$\square$
\end{proof}

\noindent  Note that  the previous result shows that we can use \textit{extrapolated} stepsize $\b \in (0, 2/L_N)$ in minibatch settings instead of the typical $\b \in (0, 2)$ used e.g. in \cite{Ned:11}.  Clearly, when $L_N<1$ we have $2/L_N >2$ and consequently, such extrapolation will accelerate convergence of the parallel algorithm. This can be also observed in simulations (see e.g. Fig. 3 below).  Moreover,  the largest decrease in  Lemma~\ref{lem-xiter1} is obtained by maximizing $\beta (2 - \beta L_N)$, that is, the optimal stepsize is $\beta = 1/L_N$. We now combine  Lemma~\ref{lem-viter} and Lemma~\ref{lem-xiter1} to provide a basic relation for the subsequent analysis.

\begin{lemma}
\label{lem-mainiter}
Consider the method in~\eqref{eq-method}, and let Assumption~\ref{asum-base} hold.
Let the stepsize $\a_k$ be such that $1 - \frac{\a_k\mu}{2}>0$ for all $k\ge0$ and stepsize $\beta \in (0, \; 2/L_N)$, with $L_N \leq 1$ defined in \eqref{eq:Lnk}.
Then, the iterates of the method~\eqref{eq-method} satisfy the following recurrence  for the optimal solution $x^*$ and for all $k\ge 0$:
\begin{eqnarray*}
&& \|v_{k+1} - x^*\|^2  +  \a_k\left(f(\Pi_X[x_k])- f^*\right) \le \left(1 - \frac{\a_k\mu}{2}\right)\|v_k - x^*\|^2\cr
&& \qquad   - \left(1 - \frac{\a_k\mu}{2}\right) \frac{\b(2 - \b L_N)}{N \, M_g^2}   \sum_{i=1}^N(g_{\o_{k}^{i}}^+(v_k))^2  +  \eta \|\Pi_X[x_k] - x_k\|^2 \cr
&& \qquad +\a_k^2 \left(1+ \frac{1}{ \eta} \right)M_f^2,
\end{eqnarray*}
where $\eta>0$ is arbitrary.
\end{lemma}

\begin{proof}
Let   $x^*\in X$ be the unique optimal solution of problem \eqref{eq-problem}. Then, we use Lemma~\ref{lem-viter} for  $\rho=\frac{1}{2}$ so that for all $k\ge0$, we have
\begin{eqnarray*}
     \|v_{k+1} - x^*\|^2  &+ & \a_k\left(f(\Pi_X[x_k])- f^*\right)
     \le \left(1 - \frac{\a_k\mu}{2}\right)\|x_k - x^*\|^2\cr
     &&+ 2 \a_k M_f\|\Pi_X[x_k] - x_k\| +\a_k^2M_f^2.
\end{eqnarray*}
Using the same reasoning as in the proof of    Lemma~\ref{lem-xiter1} for the inequality  \eqref{eq:lema4y} with $y=x^*$ gives:
\begin{align*}
\|x_{k}- x^*\|^2
& \le \|v_{k} - x^*\|^2  - \frac{\b(2-\b)}{N }  \sum_{i=1}^N   \frac{(g_{\o_{k}^{i}}^+(v_k))^2}{\|d_k^i\|^2} -  \b^2 V_N(v_k) \\
& \leq \|v_{k} - x^*\|^2  - \frac{\b(2-\b L_N)}{N \, M_g^2}  \sum_{i=1}^N(g_{\o_{k}^{i}}^+(v_k))^2.
\end{align*}
Combining the preceding two relations yields
\begin{align}
\label{eq-estvkxk1}
& \|v_{k+1} - x^*\|^2  +  \a_k\left(f(\Pi_X[x_k])- f^*\right) \le \left(1 - \frac{\a_k\mu}{2}\right)\|v_k - x^*\|^2 \\
&   - \left(1 - \frac{\a_k\mu}{2}\right) \frac{\b(2-\b L_N)}{N \, M_g^2} \sum_{i=1}^N(g_{\o_{k}^{i}}^+(v_k))^2 + 2  \a_k M_f\|\Pi_X[x_k] - x_k\| +\a_k^2M_f^2. \nonumber
\end{align}
We next approximate the  term  that is linear in $\a_k$,  i.e. $2 \a_k M_f \|\Pi_X[x_k]- x_k\|$, with a sum of two quadratic terms, one of which is in the order of $\a_k^2$, as:
\begin{align*}
   2\a_k M_f \|\Pi_X[x_k] - x_k\|
     &=2(\a_k\sqrt{\eta^{-1}}M_f) (\sqrt{\eta}\|\Pi_X[x_k]-x_k\|)\cr
     &\le   \a_k^2\eta^{-1}M_f^2 + \eta \|\Pi_X[x_k]-x_k\|^2 ,
\end{align*}
for any arbitrary $\eta>0$.   Substituting the preceding estimate in~\eqref{eq-estvkxk1},
we obtain the stated relation.  \hfill$\square$
\end{proof}

%%%%%%%%%%%%%%%%%%%%%%%%%%%%%%%%%%%%%%%%%

\subsection{Convergence rates}
\label{sec:convergence_diff}

In this section we derive the convergence rates of Algorithm \eqref{eq-method}. For this,   we first provide a recurrence  relation for the iterates in expectation,  which is the key relation for our convergence rate results. Note that $c M_g^2 \geq 1$ according to Lemma \ref{lema:c} and $L_N \in (0, 1]$.  In the sequel we  provide a detailed convergence analysis for the non-trivial case    $c M_g^2 L_N > 1$. The other case, i.e. $c M_g^2 L_N \leq 1$, implies almost sure feasibility for any $x_t$ generated by the parallel algorithm, with $t \geq 1$, and   it will be discussed in Remark \ref{rem:idealcase}.

% Thus, by increasing $c$, $M_g$ and/or $L_N$, we can always ensure that $c M_g^2 L_N> 1$. 

\begin{theorem}
\label{thm-exp-basic}
Consider the iterative process~\eqref{eq-method}, and let Assumption~\ref{asum-base} and  Assumption~\ref{asum-regularmod} hold.  Let the stepsizes    $\a_k$ be such that  $1 - \frac{\a_k\mu}{2}>0$ for all $k\ge0$  and $\beta \in (0, \; 2/L_N)$, with $L_N \leq 1$ defined in \eqref{eq:Lnk}, and  assume  $c M_g^2 L_N> 1$. Then, for the algorithm~\eqref{eq-method}, by defining $q_N= \frac{\b(2 - \b L_N)}{c M_g^2} <1$, we have almost surely for all $k\ge 0$,
\begin{eqnarray*}
&&\EXP{ \|v_{k+1} - x^*\|^2 \mid \F_{k-1}}  +  \a_k\EXP{\left(f(\Pi_X[x_k])- f^*\right) \mid \F_{k-1}} \cr
&& \le \left(1 - \frac{\a_k\mu}{2}\right)\|v_k - x^*\|^2 - \frac{1}{2}\left(1 - \frac{\a_k\mu}{2}\right) \frac{q_N}{1-q_N} \EXP{\dist^2(x_{k}, X) \mid \F_{k-1}} \cr
 && \quad + \a_k^2 \left(1 + \frac{4(1-q_N)}{q_N (2-\a_k \mu)} \right) M_f^2.
\end{eqnarray*}
\end{theorem}

\begin{proof}
From Lemma~\ref{lem-mainiter}, by taking the conditional expectation on the past $\F_{k-1}$, we have almost surely for all $k\ge0$,
\begin{eqnarray}
\label{eq-exp-nearlydone}
&&\EXP{ \|v_{k+1} - x^*\|^2\F_{k-1}}  +  \a_k\EXP{\left(f(\Pi_X[x_k])- f^*\right) \mid \F_{k-1}} \cr
&& \le \left(1 - \frac{\a_k\mu}{2}\right)\|v_k - x^*\|^2 - \left(1 - \frac{\a_k\mu}{2}\right) \frac{\b(2-\b L_N)}{N M_g^2} \sum_{i=1}^N\EXP{(g_{\o_{k}^{i}}^+(v_k))^2\mid \F_{k-1}} \cr
&& \qquad +  \eta \EXP{\|\Pi_X[x_k] - x_k\|^2 \mid \F_{k-1}}  + \a_k^2 \left( 1+ \frac{1}{\eta} \right)M_f^2,
\end{eqnarray}
where $\eta>0$ is arbitrary. By Assumption~\ref{asum-regularmod}, it follows that
       \[\EXP{(g^+_{\o_k^i}(v_k))^2\mid \F_{k-1}}\ge \frac{1}{c}\dist^2(v_k,X)\qquad
       \hbox{for all }i=1,\ldots, N.\]
Hence
\be\label{eq-expest1}
 \frac{1}{N}
 \sum_{i=1}^N\EXP{(g_{\o_{k}^{i}}^+(v_k))^2\mid \F_{k-1}}\ge \frac{1}{c}\dist^2(v_k,X).\ee
Taking the conditional expectation on the past $\F_{k-1}$ in the relation  of Lemma~\ref{lem-xiter1}, and using relation~\eqref{eq-expest1}, we obtain almost surely
\begin{align}
\label{ineq:crucial}
\EXP{\dist^2(x_{k}, X) \mid \F_{k-1}} \le \left(1-q_N\right) \dist^2(v_{k},X), 
\end{align}
where we denote
\be\label{eq-q}
 q_N= \frac{\b(2-\b L_N)}{c M_g^2}.\ee
Recall that we assume $c M_g^2 L_N > 1$, then   $q_N  < 1$ (since  $\max_{\b} \b(2-\b L_N)=1/L_N$). Hence, $1-q_N>0$.  By dividing with $1-q_N$, we further obtain
 \[\dist^2(v_{k},X)   \ge  \frac{1}{1-q_N}\EXP{\dist^2(x_{k}, X) \mid  \F_{k-1}}.\]
Substituting the preceding estimate in relation~\eqref{eq-expest1}, yields
\begin{eqnarray}
\label{eq-expest2}
\frac{1}{N}  \sum_{i=1}^N\EXP{(g_{\o_{k}^{i}}^+(v_k))^2\mid \F_{k-1}}
 & \ge & \frac{1}{c(1-q_N)}\EXP{\dist^2(x_{k}, X) \mid \F_{k-1}}.
\end{eqnarray}
We now use estimate~\eqref{eq-expest2} in relation~\eqref{eq-exp-nearlydone}, and thus obtain
\begin{eqnarray*}
&&\EXP{ \|v_{k+1} - x^*\|^2\F_{k-1}}  +  \a_k\EXP{\left(f(\Pi_X[x_k])- f^*\right) \mid \F_{k-1}} \cr
&& \le \left(1 - \frac{\a_k\mu}{2}\right)\|v_k - x^*\|^2  - \left(1 - \frac{\a_k\mu}{2}\right) \frac{\b(2-\b L_N)}{(1-q_N) c M_g^2 }
       \EXP{\dist^2(x_{k}, X) \mid \F_{k-1}} \cr
     &&\qquad + \eta \EXP{\|\Pi_X[x_k] - x_k\|^2 \mid \F_{k-1}}  + \a_k^2 \left( 1+ \frac{1}{ \eta} \right)M_f^2.
     \end{eqnarray*}
     By the definition of $q$ (see~\eqref{eq-q}), we have
     \[\frac{\b(2-\b L_N)}{(1-q_N) c M_g^2 }=\frac{q_N}{1-q_N}.\]
     Hence,
\begin{eqnarray*}
       &&\EXP{ \|v_{k+1} - x^*\|^2\F_{k-1}}  +  \a_k\EXP{\left(f(\Pi_X[x_k])- f^*\right) \mid \F_{k-1}} \cr
   &&  \le \left(1 - \frac{\a_k\mu}{2}\right)\|v_k - x^*\|^2  - \left( \left(1 - \frac{\a_k\mu}{2}\right) \frac{q_N}{1-q_N} - \eta  \right)
       \EXP{\dist^2(x_{k}, X) \mid \F_{k-1}} \cr
      &&\quad      + \a_k^2 \left( 1+ \frac{1}{ \eta} \right)M_f^2,
\end{eqnarray*}
and by letting $\eta= \frac{1}{2} \left(1 - \frac{\a_k\mu}{2}\right) \frac{q_N}{1-q_N} > 0$, the desired relation follows.  \hfill$\square$
\end{proof}

\noindent We now turn our attention to the stepsize $\a_k$. We consider $\a_k$ of the form:
\[ \a_k =  \frac{2}{\mu} \g_k \qquad\hbox{for all }k\ge0, \]
for some diminishing sequence $\g_k$ as detailed below. Indeed, for this choice, the recurrence from Theorem  \ref{thm-exp-basic}  becomes:
\begin{eqnarray}\label{eq-relgkexp0}
&&\EXP{ \|v_{k+1} - x^*\|^2 \mid \F_{k-1}}  +  \frac{2}{\mu}\g_k\EXP{\left(f(\Pi_X[x_k])- f^*\right) \mid \F_{k-1}} \cr
&&       \le \left(1 - \g_k\right)\|v_k - x^*\|^2  - \frac{1}{2}\left(1 - \g_k\right) \frac{q_N}{1-q_N}       \EXP{\dist^2(x_{k}, X) \mid \F_{k-1}} \\
        && \quad       +\frac{4}{\mu^2}\g_k^2 \left(1+ \frac{4(1-q_N)}{q_N (2-2\g_k)}\right)M_f^2, \nonumber
\end{eqnarray}
where recall that $q_N= \frac{\b(2-\b L_N)}{c M_g^2}$. Let  $\g_k$ be given by
\[ \g_k = \frac{2}{k+1}, \quad \hbox{hence the stepsize} \quad \a_k = \frac{4}{\mu(k+1)},  \;\;\; \forall  k \ge 0. \]
Since the sequence $\g_k$ is decreasing, we have
       \[\g_k\le \frac{2}{3}\qquad \hbox{for all }k\ge 1,\]
       implying that
       \[1-\g_k\ge \frac{1}{3}\qquad \hbox{for all }k\ge 1.\]
Using this estimate in~\eqref{eq-relgkexp0}, we obtain
\begin{eqnarray}\label{eq-relgkexp1}
       &&\EXP{ \|v_{k+1} - x^*\|^2\F_{k-1}}  +  \frac{2}{\mu}\g_k\EXP{\left(f(\Pi_X[x_k])- f^*\right) \mid \F_{k-1}} \cr
      &&  \le \left(1 - \g_k\right)\|v_k - x^*\|^2  - \frac{1}{6}\frac{q_N}{1-q_N}
       \EXP{\dist^2(x_{k}, X) \mid \F_{k-1}} \\
        && \quad       +\frac{4}{\mu^2}\g_k^2 \left(1+ \frac{6(1-q_N)}{q_N}\right)M_f^2. \nonumber
       \end{eqnarray}

\noindent  Next, we note that
 \[\frac{1- \g_k}{\g_k^2}\le \frac{1} {\g_{k-1}^2}\qquad\qquad\hbox{for all }k\ge1.\]
Dividing~\eqref{eq-relgkexp1} by $\g_k^2$ and using the preceding inequality we have for all $k\ge1$, after taking total expectations and rearranging terms:
\begin{eqnarray*}
%\label{eq-relgkexp2}
&&\g_k^{-2} \EXP{ \|v_{k+1} - x^*\|^2}  +  \frac{2}{\mu}\g_k^{-1}\EXP{\left(f(\Pi_X[x_k])- f^*\right)} + \frac{\g_k^{-2}}{6}\frac{q_N}{1-q_N}\EXP{\dist^2(x_{k}, X)} \cr
&&   \le \g_{k-1}^{-2}\EXP{\|v_k - x^*\|^2} +\frac{4}{\mu^2}\left(1+ \frac{6(1-q_N)}{q_N}\right)M_f^2.
\end{eqnarray*}
Summing these over $k=1,\ldots, t$, for some $t>0$, we obtain
\begin{eqnarray}
\label{eq-relgkexp3}
 &&\g_t^{-2}
       \EXP{ \|v_{t+1} - x^*\|^2}
       +  \frac{2}{\mu}\sum_{k=1}^t\g_k^{-1}\EXP{\left(f(\Pi_X[x_k])- f^*\right)} \\
&&        + \frac{1}{6}\frac{q_N}{1-q_N}\sum_{k=1}^t \g_k^{-2}\EXP{\dist^2(x_{k}, X)} \le \g_{0}^{-2}\EXP{\|v_1 - x^*\|^2} +t \frac{4}{\mu^2}\left(1+ \frac{6(1-q_N)}{q_N}\right)M_f^2. \nonumber
       \end{eqnarray}
Using the definition of $\g_k$, \eqref{eq-relgkexp3} implies
\begin{eqnarray}\label{eq-relgkexp4}
       &&\frac{(t+1)^2}{4}
       \EXP{ \|v_{t+1} - x^*\|^2}
       +  \frac{1}{\mu}\sum_{k=1}^t (k+1)\EXP{\left(f(\Pi_X[x_k])- f^*\right)} \cr
       &&+ \frac{q_N}{24(1 \!-\! q_N)} \! \sum_{k=1}^t (k+1)^2\EXP{\dist^2(x_{k}, X)} \!\le\! \frac{1}{4}\EXP{\|v_1 - x^*\|^2}
       +\! \frac{4 t}{\mu^2} \! \left(1+ \frac{6(1-q_N)}{q_N}\right) \! M_f^2. \nonumber
       \end{eqnarray}
We finally obtain by the linearity of the expectation operation:
        \begin{eqnarray}\label{eq-relgkexp5}
       &&\frac{(t+1)^2}{4}
       \EXP{ \|v_{t+1} - x^*\|^2}
       +  \frac{1}{(t+1)\mu}\EXP{\sum_{k=1}^t (k+1)^2\left(f(\Pi_X[x_k])- f^*\right)} \cr
        && \qquad + \frac{q_N}{24(1-q_N)}\EXP{\sum_{k=1}^t (k+1)^2 \|x_{k} - \Pi_X[x_k]\|^2} \\
        && \le \frac{1}{4}\EXP{\|v_1 - x^*\|^2} +\frac{4 t}{\mu^2}\left(1+ \frac{6(1-q_N)}{q_N}\right)M_f^2. \nonumber
       \end{eqnarray}
Define  for $t\ge 1$ the sum
       \[S_t=\sum_{k=1}^t (k+1)^2 \sim {\cal O}(t^3).\]
Define also the following weighted averages (convex combinations)
 \begin{align}
\label{eq:paravseq}
 \hat x_t=\sum_{k=1}^t a_k\,x_k, \qquad  \hat w_t= \sum_{k=1}^t a_k\,\Pi_X[x_k],
 \end{align}
with $a_k=\frac{(k+1)^2 }{S_t}$, hence satisfying $\sum_{k=1}^t a_k=1$.
Using convexity of the  function $f$ and of the norm-squared, we have
       \begin{eqnarray}\label{eq-relgkexp60}
       &&\frac{(t+1)^2}{4}
       \EXP{ \|v_{t+1} - x^*\|^2}
       +  \frac{S_t}{(t+1)\mu}\EXP{\left(f(\hat w_t)- f^*\right)} + \frac{q_N S_t}{24(1-q_N)}\EXP{\|\hat w_t - \hat x_t\|^2} \cr
       &&\le \frac{1}{4}\EXP{\|v_1 - x^*\|^2} +\frac{4 t}{\mu^2}\left(1+ \frac{6(1-q_N)}{q_N}\right)M_f^2.
       \end{eqnarray}
If we define $b_N^p = q_N(1-q_N)^{-1} = (1-q_N)^{-1} -1$, then \eqref{eq-relgkexp60}  becomes:
\begin{eqnarray}\label{eq-relgkexp6}
       &&\frac{(t+1)^2}{4}
       \EXP{ \|v_{t+1} - x^*\|^2}
       +  \frac{S_t}{(t+1)\mu}\EXP{\left(f(\hat w_t)- f^*\right)} + \frac{b_N^p S_t}{24} \EXP{\|\hat w_t - \hat x_t\|^2} \cr
       &&\le \frac{1}{4}\EXP{\|v_1 - x^*\|^2} +\frac{4 t}{\mu^2}\left(1+ \frac{6}{b_N^p}\right)M_f^2.
       \end{eqnarray}

\noindent Next theorem summarizes the convergence rates followed from the previous discussion. For simplicity of the exposition, we omit the constants and express the rates only in terms of the dominant powers of $t$:
\begin{theorem}
\label{theorem:mainparalg}
Let Assumption~\ref{asum-base} and  Assumption~\ref{asum-regularmod} hold and the stepsizes    $\a_k = \frac{4}{\mu(k+1)}$  and $\beta \in (0, \; 2/L_N)$, with $L_N \leq 1$ defined in \eqref{eq:Lnk}. Let also  assume  $c M_g^2 L_N> 1$. Then,  $q_N= \frac{\b(2-\b L_N)}{c M_g^2} <1$ and $b_N^p = (1-q_N)^{-1} -1$. Moreover, the following sublinear  rates for suboptimality and feasibility violation hold for the average sequence $ \hat x_t$ generated by  the parallel  algorithm~\eqref{eq-method}:
\[    \EXP{|f(\hat x_t) - f^*| } \leq {\cal O} \left( \frac{1}{t}  + \frac{1}{\sqrt{b_N^p} t} \right), \quad  \EXP{\dist_X(\hat x_t)}  \leq  {\cal O} \left( \frac{1}{\sqrt{b_N^p}t} \right).   \]
\end{theorem}

\begin{proof}
From the recurrence  \eqref{eq-relgkexp6}, omitting the constants but keeping the terms depending on $b_N^p =  (1-q_N)^{-1} -1$, we get the following convergence rates in terms of these weighted averages $\hat w_t$ and $\hat x_t$:
\[ \EXP{f(\hat w_t) - f^*)} \leq {\cal O} \left( \frac{1}{t} + \frac{1}{b_N^p t} \right)  \quad \text{and} \quad  \EXP{\|\hat w_t - \hat x_t\|^2}   \leq {\cal O} \left( \frac{1}{b_N^p t^2} \right). \]
Since   $\hat w_t\in X$ and using  the  Jensen's inequality we get the following convergence rate for the feasibility violation of the constraints:
\[   \EXP{\dist_X(\hat x_t)}  \leq   \EXP{\|\hat w_t - \hat x_t\|} \le \sqrt{\EXP{\|\hat w_t - \hat x_t\|^2} } \leq {\cal O} \left( \frac{1}{\sqrt{b_N^p} t} \right).  \]
Since $\hat x_t\in Y$ and $\hat w_t\in X\subset Y$, by the subgradient boundedness of $f$ on $Y$,   it follows that
\[ \EXP{ |f(\hat x_t)- f(\hat w_t)|} \le M_f  \EXP{ \|\hat x_t -\hat w_t\| } \leq {\cal O} \left( \frac{1}{\sqrt{b_N^p} t} \right), \]
which combined with $\EXP{f(\hat w_t) - f^*)} \leq {\cal O} \left( \frac{1}{t} + \frac{1}{b_N^p t} \right)$, yields also the following convergence rate for suboptimality
\[   \EXP{|f(\hat x_t) - f^*| } \leq {\cal O} \left(  \frac{1}{t}  + \frac{1}{\sqrt{b_N^p} t}  \right),  \]
which proves our theorem. \hfill$\square$
\end{proof}

\noindent We observe that the convergence estimate for the feasibility violation depends explicitly on the minibatch size $N$ via the key parameter $L_N$.  For the optimal stepsize $\beta = 1/L_N$ we get $q_N = 1/c M_g^2 L_N$ and $b_N^p = 1/ (c M_g^2 L_N -1)$. Hence, $b_N^p$ is large provided that $L_N \ll 1$ (small). Note that if $L_N=1$, then $b_N^p$ does not depend on $N$ and hence complexity does not improve with minibatch size $N$. However, as long as $L_N <1$ (and it can be also the case that $L_N \sim 0$), then $b_N^p$ becomes  large, which shows that minibatching  improves complexity.  \textit{To the best of our knowledge, this is the first time that a subgradient method with random minibatch  feasibility updates  is shown to be better than its non-minibatch  variant. We have identified $L_N$ as the key quantity determining whether minibatching helps ($L_N < 1$) or not ($L_N = 1$), and how much (the smaller $L_N$, the more it helps)}. Note also  that the suboptimality estimate contains a term which does not depend on the minibatch size $N$ as it happens for feasibility violation  estimate. This is natural, since the minibatch feasibility steps have no effect on the minimization step of the objective function.

\begin{remark}
Note that the convergence rates  ${\cal O} \left( \frac{1}{t} \right)$ for feasibility and suboptimality  are known to be optimal for the stochastic subgradient method for solving the optimization problem \eqref{eq-problem} under Assumption \ref{asum-base}, see \cite{NemYud:83,Nes:04}. Moreover, the iterative process \eqref{eq-method} does not require knowledge of   the subgradient norm bounds $M_f$ and $M_g$ from Assumption \ref{asum-base}, nor the constant $c$ from Assumption~\ref{asum-regularmod}.  These values are only affecting the constants in the convergence rates, they are not needed for the stepsize selection. The stepsize $\a_k$ requires only  knowledge of some  estimate of the strong convexity constant~$\mu$. Moreover, since $L_N \leq 1$, we can use e.g.,  stepsize $\b \in (0, \; 2) \subseteq  (0, \;  2/L_N)$. Of course, a larger stepsize $\beta$ leads to a faster convergence. Hence, if $L_N < 1$ and it  can be computed, then we should choose an extrapolated steplength $\beta = (2 -\delta)/L_N$ for some $\delta \in (0, 2)$ small. When $L_N$ cannot be computed explicitly, we propose to approximate it online with $L_N^k$, and use  at each iteration an adaptive extrapolated  stepsize  $\beta_k$ of the form $\beta_k =  (2-\delta)/L_N^k$ for some $\delta \in (0, \;  2)$  (see also the discussion from Section \ref{sec:extensions}, equation \eqref{eq:betak}). 
\end{remark}

\begin{remark}
\label{rem:idealcase}
The convergence rates from Theorem \ref{theorem:mainparalg} hold for the non-trivial case $q_N = \frac{\b(2-\b L_N)}{c M_g^2} <1$. Note that the inequality  $q_N<1$ is always satisfied,  provided that  $c M_g^2 L_N> 1$. On the other hand, the  case $q_N \geq 1$ (e.g.,  $c M_g^2 L_N \leq  1$ and $\b = 1/L_N$)  turns out to be  the ideal case, since then we have from  \eqref{ineq:crucial} that  $ \EXP{\dist^2(x_{k}, X) \mid \F_{k-1}} \le  0  \quad \forall k \geq 1$.   Therefore, in this ideal case we achieve almost sure feasibility for the sequence $x_t$ generated by the parallel  algorithm (see \eqref{eq-method}) after one step:
 \[  \EXP{\dist^2(x_{t}, X)} =  0  \quad \forall t \geq 1.   \] 
Using this feasibility relation in the same derivations from Section  
\ref{sec:convergence_diff} we also get  a suboptimality estimate for the average 
sequence  $ \hat x_t$ as in Theorem \ref{theorem:mainparalg}:
\[    \EXP{|f(\hat x_t) - f^*| } \leq {\cal O} \left( \frac{1}{t}  \right)  \quad \forall t \geq 1.   \]
Clearly, from Jensen's inequality we also have almost sure feasibility for the average sequence  $ \hat x_t$:
 \[  \EXP{\dist^2(\hat x_{t}, X)} =  0  \quad \forall t \geq 1.   \] 
 We skip these details since the proof is the same as for the non-trivial case. 
\end{remark}

%%%%%%%%%%%%%%%%%%%%%%%%%%%%%%%%%%%%

\subsection{Example of functional constraints having $L_N<1$}
\label{sec:Lnk}
\noindent Let us recall the definition of  the parameters $L_N^k$ and $L_N$ from \eqref{eq:Lnk}:
\begin{align*}
L_N^k =   \left\| \frac{1}{N}\sum_{i=1}^N\frac{g_{\o_{k}^{i}}^+(v_k)}{\|d_k^i\|^2}\, d_k^i\right\|^2 \Big{/} \frac{1}{N} \sum_{i=1}^N\frac{(g_{\o_{k}^{i}}^+(v_k))^2}{\|d_k^i\|^2}   \quad \text{and}  \quad L_N = \max_{k \geq 0} L_N^k.
\end{align*}
From Jensen's inequality we have $L_N^k \leq 1$ and consequently $L_N \leq 1$. On the other hand,   Theorem \ref{theorem:mainparalg} shows that $L_N \ll 1$ is beneficial for a subgradient scheme with minibatch feasibility updates.  In this section we provide  an example of   functional constraints $g_\o$ for which  $L_N < 1$. Let us consider $m$ linear inequality constraints for the convex problem \eqref{eq-problem}:
\[  g_\o(x) = a_\o^T x + b_\o \leq 0  \quad  \forall  \o \in \A=\{1, 2, \cdots, m\}.   \]
Without loss of generality we assume $\|a_\o\|  =1$ for all $\o$. Let us define the matrix $A= [a_1 \cdots a_m]^T$ and the subset of indexes selected at the current  iteration $J_k =\{ \o^1_k \cdots \o^N_k \} \subset \A$. We also denote $J_k^+ = \{ \o \in J_k:  a_\o^T v_k + b_\o > 0  \}$ and denote $A_{J_k^+}$ the submatrix of $A$ having the rows indexed in the set $J_k^+$. With these notations and using that  $\|a_\o\|  =1$ for all $\o$, then  $L_N^k$ can be written explicitly as (assuming that $|J_k^+| \geq 1$):
\begin{align*}
L_N^k & =   \left\| \frac{1}{N}\sum_{\o \in J_k^+} \frac{g_{\o}^+(v_k)}{\|d_k^\o\|^2}\, d_k^\o\right\|^2 \Big{/} \frac{1}{N} \sum_{\o \in J_k^+} \frac{(g_{\o}^+(v_k))^2}{\|d_k^\o\|^2}  \\
& = \left\| \sum_{\o \in J_k^+} (a_\o v_k + b_\o) a_\o \right \|^2 \Big{/} N \sum_{\o \in J_k^+} (a_\o v_k + b_\o)^2  \\
& = \left\|   A_{J_k^+}^T ( A_{J_k^+} v_k + b_{J_k^+})  \right\|^2/ N \| A_{J_k^+} v_k + b_{J_k^+} \|^2 \\
& \leq \frac{\lambda_{\max} (A_{J_k^+} A_{J_k^+}^T)}{N}   \leq \frac{\lambda_{\max} (A_{J_k} A_{J_k}^T)}{N} \\
& < \frac{\text{Trace}(A_{J_k} A_{J_k}^T)}{N} =  1  \quad \forall k,
\end{align*}
where the first inequality follows from the definition of the maximal eigenvalue $\lambda_{\max}$ of a matrix, the second inequality follows from the fact that $J_k^+ \subseteq J_k$, and the third  inequality holds strictly provided that the submatrix $A_{J_k}$ has at least rank two.   In conclusion, if the matrix $A$ has e.g. full row rank and consider a sampling of $J_k$ based on a given probability $\textbf{P}$, then $L_N$ satisfies:
\begin{align}
\label{eq:LNlin}
L_N = \max_{ J \in 2^\A, |J| = N, J \sim \textbf{P}}  \frac{\lambda_{\max} (A_{J} A_{J}^T)}{N} <1.
\end{align}
Note that for particular sampling rules we can compute $L_N$  efficiently, such as when we consider a uniform distribution over a fixed partition of $\A = \cup_{i=1}^\ell J_i$ of equal size. The reader may find other examples of functional constraints satisfying $L_N<1$ and we believe that this paper opens a window of opportunities for algorithmic research in this direction.

%%%%%%%%%%%%%%%%%%%%%%%%%%%%%%%%%%%%%%%
%%%%%%%%%%%%%%%%%%%%%%%%%%%%%%%%%%%%%%%%%

\section{Sequential random minibatch subgradient algorithm}
\label{sec:convergence_diff_seq}
In this section we consider a \textit{sequential} variant of the algorithm  \eqref{eq-method} defined in terms of the following iterative process:

\begin{center}
\framebox{
\parbox{11 cm}{
\begin{center}
\textbf{ Algorithm (sequential case) }
\end{center}
Choose $x^0 \in Y$, minibatch size $N \ge 1$, and 
stepsizes $\alpha_k >0$ and $\b > 0$. For $k \geq 1$ repeat:
\begin{enumerate}
\item Draw $N$  samples $J_k =\{ \o^k_1, \cdots, \o^k_N \}  \sim \textbf{P}$.
\item Compute the following updates: 
\begin{subequations}
\label{eq-method-sequential}
 \begin{eqnarray}
   && v_k = \Pi_Y[x_{k-1} - \a_{k-1} s_f(x_{k-1})],\label{eq-vup}\\
   && z_{k}^0=v_k,  \;  z_{k}^i   =\Pi_Y \! \left[z_{k}^{i-1} - \beta\, \frac{g^+_{\o_{k}^{i}}(z_{k}^{i-1})}{\|d_{k}^{i}\|^2}\, d_{k}^{i} \right]  \; \hbox{for } i\!=\!1\!:\!N, \label{eq-zseq}\\
   && x_k= z_{k}^{N}\label{eq-xup}.
  \end{eqnarray}
  \end{subequations}
\end{enumerate}
}}
\end{center}

\noindent This method takes, as for the parallel variant, one subgradient step for the objective function,  followed by $N$ \textit{sequential}  feasibility updates.  As before,  the vector $d_{k}^{i}$ is chosen as $d_{k}^{i} \in \partial g^+_{\o_{k}^{i}}(z_{k}^{i-1})$ if $g^+_{\o_{k}^{i}}(v_{k})>0$,  and $d_{k}^{i}=d$  for some $d \ne 0$ if $g^+_{\o_{k}^{i}}(z_{k}^{i-1})=0$.  Note that in this variant, the feasibility updates use the projection on $Y$ in order to confine the intermittent iterates $z_k^i$  and $x_k$ to the set $Y$, where $g_\o$'s and $f$ (for the last step)  are assumed to have uniformly bounded subgradients.

\noindent In this section we analyze the convergence properties of this new algorithm \eqref{eq-method-sequential}. Given $x_{k-1}$, the update of $v_k$ is the same as in the parallel method \eqref{eq-method}, thus Lemma~\ref{lem-viter}  still applies here.  We  need an analog of Lemma~\ref{lem-xiter1}.

\begin{lemma}
\label{lem-xiter-sequential}
Let Assumptions~\ref{asum-base}(a) and~\ref{asum-base}(c) hold. Let $x_{k}$ be generated by algorithm \eqref{eq-method-sequential} with $\beta \in (0, \; 2)$. Then, the following relations are valid:
\begin{align*}
& \dist^2(z_{k}^i,X) \le  \dist^2(z_{k}^{i-1},X)  - \frac{\b(2-\beta)}{M_g^2}\,(g_{\o_{k}^{i}}^+(z_k^{i-1}))^2   \hbox{ for all } i=1,\ldots,N,   \\
& \|x_k - y\|^2 \le  \|v_{k} - y\|^2 -\frac{\b(2-\beta)}{M_g^2}\,\sum_{i=1}^N (g_{\o_{k}^{i}}^+(z_k^{i-1}))^2  \qquad \hbox{for all } y \in X,\\
& \dist^2(x_k,X) \le  \dist^2(v_{k},X) - \frac{\b(2-\beta)}{M_g^2}\, \sum_{i=1}^N (g_{\o_{k}^{i}}^+(z_k^{i-1}))^2 \qquad \hbox{for all } k \geq 1.
\end{align*}
\end{lemma}

\begin{proof}
We start with the definition of $z_{k}^i$ in~\eqref{eq-zseq} and Lemma~\ref{lemma:basiter}, with $Z=Y$. Thus, we obtain   for all $y\in X$ (which satisfies $g_{\o_{k}^i}^+(y)=0$
      for any realization of $\o_{k}^i$) and for all $i=1,\ldots,N$,
      \[\|z_{k}^i - y\|^2 \le  \|z_{k}^{i-1} - y\|^2
       -\b(2-\beta)\,\frac{(g_{\o_{k}^{i}}^+(z_k^{i-1}))^2}{\|d_k^{i}\|^2}.\]
       By using $\|d_k^i\|^2 \le M_g^2$, we have for all $i=1,\ldots,N$,
       \be\label{eq-zse}
       \|z_{k}^i - y\|^2 \le  \|z_{k}^{i-1} - y\|^2
       -\frac{\b(2-\beta)}{M_g^2} \, (g_{\o_{k}^{i}}^+(z_k^{i-1}))^2.\ee
       The distance relation for $z$-iterates follows by taking the minimum over $y\in X$ on
       both sides of inequality~\eqref{eq-zse}. By summing relations~\eqref{eq-zse} over $i=1,\ldots,N$,  and by using $z^0_k=v_k$ and $z^N_k=x_k$, we obtain for any $y\in X$,
       \[\|x_k - y\|^2 \le  \|v_{k} - y\|^2
       -\frac{\b(2-\beta)}{M_g^2}\,
       \sum_{i=1}^N (g_{\o_{k}^{i}}^+(z_k^{i-1}))^2.\]
The distance relation follows by taking the minimum over $y\in X$ on
       both sides of the preceding inequality.  \hfill$\square$
       \end{proof}

\noindent Taking $\rho=1/2$ in  Lemma~\ref{lem-viter} we get:
\begin{eqnarray*}
     \|v_{k+1} - x^*\|^2  &+ & \a_k\left(f(\Pi_X[x_k])- f^*\right)
     \le \left(1  - \frac{\a_k\mu}{2}\right) \|x_k - x^*\|^2\cr
     && + 2 \a_k M_f \|\Pi_X[x_k] - x_k\| +\a_k^2M_f^2,
     \end{eqnarray*}
and  using the inequality for $\|x_k-y\|^2$ from Lemma~\ref{lem-xiter-sequential} in $y=x^*$, yields:
\begin{eqnarray}
\label{eq:seqalg1}
&& \|v_{k+1}-x^*\|^2  +  \a_k\left(f(\Pi_X[x_k])- f^*\right)
     \le \left(1- \frac{\a_k\mu}{2}\right) \|v_k - x^*\|^2 \\
     &&-\left(1- \frac{\a_k\mu}{2}\right)
     \frac{\b(2-\beta)}{M_g^2}\,\sum_{i=1}^N (g_{\o_{k}^{i}}^+(z_k^{i-1}))^2  + 2  \a_k M_f\|\Pi_X[x_k] - x_k\| +\a_k^2M_f^2. \nonumber
\end{eqnarray}

\noindent Taking the conditional expectation on $\F_{k-1}$ and $z_k^{i-1}$, and using
Assumption~\ref{asum-regularmod}, give
\[\EXP{g_{\o_{k}^{i}}^+(z_k^{i-1}))^2\mid\F_{k-1}, z_k^{i-1}}
     \ge \frac{1}{c}\dist(z_k^{i-1},X).\]
Using the iterated expectation rule, we obtain
\be\label{eq-startexpest}
\EXP{g_{\o_{k}^{i}}^+(z_k^{i-1}))^2\mid\F_{k-1}}
=  \EXP{\EXP{g_{\o_{k}^{i}}^+(z_k^{i-1}))^2\mid\F_{k-1}, z_k^{i-1}}}
\ge \frac{1}{c}\EXP{\dist(z_k^{i-1},X)},\ee
which, when combined with the distance relation of Lemma  \ref{lem-xiter-sequential} gives
     for all $i=1,\ldots,N$
\[  \EXP{\dist^2(z_{k}^i,X)   \mid\F_{k-1}} \le \left( 1-\frac{\b(2-\beta)}{c M_g^2}\right)
      \EXP{\dist^2(z_{k}^{i-1},X) \mid\F_{k-1}}.\]

\noindent Recall  that $c M_g^2 \geq 1$ according to Lemma \ref{lema:c}.  In the subsequent analysis we consider the non-trivial case $cM_g^2 >1$. The ideal case  $cM_g^2 =1$ will allow to get  feasibility in expectation  in one step and obtain a similar convergence rate result as in Remark  \ref{rem:idealcase}.  Hence, using the definition of $x_k$, i.e., $x_k=z_k^N$, and letting $q=\frac{\b(2-\beta)}{c M_g^2} \in (0,1)$ (since we assume  $cM_g^2 >1$ and $\b \in (0,2)$),  we have for all $i=1,\ldots,N$,
      \[\EXP{\dist^2(x_k,X)   \mid\F_{k-1}}\le
      (1-q)^{N-i+1} \EXP{\dist^2(z_{k}^{i-1},X) \mid\F_{k-1}},\]
      implying that for all $i=1,\ldots,N$,
      \be\label{eq-almexpest}
      \EXP{\dist^2(z_{k}^{i-1},X) \mid\F_{k-1}}\ge
      \frac{1}{(1-q)^{N-i+1}}\EXP{\dist^2(x_k,X)\mid\F_{k-1}},\ee
     From~\eqref{eq-startexpest} and~\eqref{eq-almexpest}
     for all $i=1,\ldots,N$,
     \[ \EXP{g_{\o_{k}^{i}}^+(z_k^{i-1}))^2\mid\F_{k-1}}
     \ge \frac{1}{c}\frac{1}{(1-q)^{N-i+1}}\EXP{\dist^2(x_k,X)\mid\F_{k-1}}.\]
     By summing over $i$
     \[\sum_{i=1}^N \EXP{g_{\o_{k}^{i}}^+(z_k^{i-1}))^2\mid\F_{k-1}}
     \ge \frac{1}{c}
     \left(\sum_{i=1}^N \frac{1}{(1-q)^{N-i+1}}\right)\EXP{\dist^2(x_k,X)\mid\F_{k-1}}.\]
However,
     \[\sum_{i=1}^N \frac{1}{(1-q)^{N-i+1}}
     =\frac{1}{(1-q)^{N+1}}\sum_{i=1}^N (1-q)^i
     =\frac{1- (1-q)^N}{q(1-q)^N}.\]
Finally, we get
      \[ \sum_{i=1}^N \EXP{g_{\o_{k}^{i}}^+(z_k^{i-1}))^2\mid\F_{k-1}}
     \ge \frac{1}{c}\frac{\left(1- (1-q)^N\right)}{q(1-q)^N}  \EXP{\dist^2(x_k,X)\mid\F_{k-1}}  \]
and consequently
\begin{align*}
& \frac{\b(2-\beta)}{M_g^2} \, \sum_{i=1}^N \EXP{(g_{\o_{k}^{i}}^+(z_k^{i-1}))^2 \mid \F_{k-1}} \\
     & \ge \frac{\b(2-\beta)}{c M_g^2 } \frac{\left(1- (1-q)^N\right)}{q(1-q)^N}
     \EXP{\dist^2(x_k,X)\mid\F_{k-1}}\cr
     &= q\frac{\left(1- (1-q)^N\right)}{q(1-q)^N}
     \EXP{\dist^2(x_k,X)\mid\F_{k-1}}\cr
     &=\frac{\left(1- (1-q)^N\right)}{(1-q)^N}
     \EXP{\dist^2(x_k,X)\mid\F_{k-1}}\cr
     &=\left( (1-q)^{-N}-1\right)  \EXP{\dist^2(x_k,X)\mid\F_{k-1}}.
     \end{align*}

\noindent  Let us denote   $b_N^s= (1-q)^{-N}-1$. It is clear that $b_N^s \to \infty$ as $N \to \infty$. Taking expectation in \eqref{eq:seqalg1} and using the previous inequality we get an analog of Lemma \ref{lem-mainiter}:
\begin{eqnarray*}
&&\EXP{ \|v_{k+1} - x^*\|^2 \mid \F_{k-1}}  +  \a_k\EXP{\left(f(\Pi_X[x_k])- f^*\right) \mid \F_{k-1}} \cr
&& \le \left(1 - \frac{\a_k\mu}{2}\right)\|v_k - x^*\|^2  - \left( \left(1 - \frac{\a_k\mu}{2}\right) b_N^s - \eta \right) \EXP{\dist^2(x_{k}, X) \mid \F_{k-1}} \cr
&& \qquad + \a_k^2 \left( 1+ \frac{1}{ \eta} \right)M_f^2,
\end{eqnarray*}
for any $\eta >0$.   Let us consider the same stepsize  as for the parallel scheme, i.e. $\a_k = \frac{2}{\mu} \g_k$,  choose $\eta= \frac{1}{2}\left(1 - \frac{\a_k\mu}{2}\right) b_N^s > 0$, and take the full expectation, to get the following recurrence (analog to Theorem \ref{thm-exp-basic}):
 \begin{eqnarray*}
% \label{eq-relgkexp-seqalg}
&&\EXP{ \|v_{k+1} - x^*\|^2}  +  \frac{2}{\mu} \g_k \EXP{\left(f(\Pi_X[x_k])- f^*\right)}    \le \left(1 - \g_k \right) \EXP{\|v_k - x^*\|^2}  \cr
&&  - \frac{1}{2}\left(1 - \g_k\right) b_N^s  \EXP{\dist^2(x_{k}, X)}   +\frac{4}{\mu^2}\g_k^2 \left(1+ \frac{2}{(1-\g_k) b_N^s}\right)M_f^2.
\end{eqnarray*}
Using now $\g_k = \frac{2}{k+1}$, then   $1-\g_k\ge \frac{1}{3}$ and we get:
\begin{eqnarray}
 \label{eq-relgkexp-seqalg}
&&\EXP{ \|v_{k+1} - x^*\|^2}  +  \frac{2}{\mu} \g_k \EXP{\left(f(\Pi_X[x_k])- f^*\right)}    \le \left(1 - \g_k \right) \EXP{\|v_k - x^*\|^2}  \cr
&&  - \frac{1}{6} b_N^s
       \EXP{\dist^2(x_{k}, X)}   +\frac{4}{\mu^2}\g_k^2 \left(1+ \frac{6}{b_N^s}\right)M_f^2.
\end{eqnarray}

\noindent  Since, $\frac{1- \g_k}{\g_k^2}\le \frac{1} {\g_{k-1}^2}$ for all  $k\ge1$,  dividing~\eqref{eq-relgkexp-seqalg} by $\g_k^2$ and using the preceding inequality we have for all $k\ge1$:
\begin{eqnarray*}
%\label{eq-relgkexp2}
&&\g_k^{-2} \EXP{ \|v_{k+1} - x^*\|^2}  +  \frac{2}{\mu}\g_k^{-1}\EXP{\left(f(\Pi_X[x_k])- f^*\right)} \cr
&& + \frac{\g_k^{-2}}{6} b_N^s \EXP{\dist^2(x_{k}, X)}  \le \g_{k-1}^{-2}\EXP{\|v_k - x^*\|^2} +\frac{4}{\mu^2}\left(1+ \frac{6}{b_N^s}\right)M_f^2.
\end{eqnarray*}
Summing these over $k=1,\ldots, t$, for some $t>0$, we obtain the following recurrence relation for the algorithm \eqref{eq-method-sequential}:
       \begin{eqnarray}
       \label{eq:mainseq0}
       &&\g_t^{-2}
       \EXP{ \|v_{t+1} - x^*\|^2}
       +  \frac{2}{\mu}\sum_{k=1}^t\g_k^{-1}\EXP{\left(f(\Pi_X[x_k])- f^*\right)} \\
      &&  + \frac{b_N^s}{6} \sum_{k=1}^t  \g_k^{-2}\EXP{\dist^2(x_{k}, X)}  \le \g_{0}^{-2}\EXP{\|v_1 - x^*\|^2}
       +t \frac{4}{\mu^2} \left(1+ \frac{6}{b_N^s} \right)M_f^2. \nonumber
       \end{eqnarray}
Using the same definition for the weighted averages $\hat w_t$ and $\hat x_t$ from \eqref{eq:paravseq} and $\g_k =\frac{2}{k+1}$ in  \eqref{eq:mainseq0}, we get the main recurrence for the sequential  variant \eqref{eq-method-sequential}:
\begin{eqnarray}\label{eq:mainseq}
       &&\frac{(t+1)^2}{4}
       \EXP{ \|v_{t+1} - x^*\|^2}
       +  \frac{S_t}{(t+1)\mu}\EXP{\left(f(\hat w_t)- f^*\right)} + \frac{b_N^s S_t}{24} \EXP{\|\hat w_t - \hat x_t\|^2} \cr
       &&\le \frac{1}{4}\EXP{\|v_1 - x^*\|^2} +\frac{4 t}{\mu^2}\left(1+ \frac{6}{b_N^s}\right)M_f^2.
       \end{eqnarray}

\noindent Next theorem summarizes the convergence rates that follow from the  recurrence relation  \eqref{eq:mainseq} of the sequential  algorithm \eqref{eq-method-sequential}.
\begin{theorem}
Let Assumption~\ref{asum-base} and  Assumption~\ref{asum-regularmod} hold and the stepsizes    $\b\in(0,2)$ and $\a_k = \frac{4}{\mu(k+1)}$. Let also $q= \frac{\b(2-\b)}{c M_g^2} <1$ and $b_N^s= (1-q)^{-N}-1$. Then, the following sublinear  rates for suboptimality and feasibility violation hold for the average sequence $ \hat x_t$ from \eqref{eq:paravseq} generated by  the sequential algorithm~\eqref{eq-method-sequential}:
\[    \EXP{|f(\hat x_t) - f^*| } \leq {\cal O} \left( \frac{1}{t}  + \frac{1}{\sqrt{b_N^s} t} \right), \quad  \EXP{\dist_X(\hat x_t)}  \leq {\cal O} \left( \frac{1}{\sqrt{b_N^s} t} \right).   \]
\end{theorem}

\begin{proof}
 Defining the same average sequences $\hat w_t$ and $\hat x_t$ as in \eqref{eq:paravseq}, we get the following convergence rates (omitting the constants but keeping the terms depending on $b_N^s$):
\[ \EXP{f(\hat w_t) - f^*)} \leq {\cal O} \left( \frac{1}{t} + \frac{1}{b_N^s t} \right), \qquad \EXP{\|\hat w_t - \hat x_t\|^2}   \leq {\cal O} \left( \frac{1}{b_N^s t^2 } \right). \]
Hence, we get the following convergence rate for the feasibility violation of the constraints that depends explicitly on the minibatch size $N$ via the term $b_N$:
\[   \EXP{\dist_X^2(\hat x_t)}  \leq    {\cal O} \left( \frac{1}{b_N^s t^2} \right).  \]
Using the same reasoning as in the proof of Theorem \ref{theorem:mainparalg},  we  also get the following convergence rate for suboptimality:
\[   \EXP{|f(\hat x_t) - f^*| } \leq {\cal O} \left( \frac{1}{t}  + \frac{1}{\sqrt{b_N^s} t} \right),  \]
which proves the statements of the theorem.  \hfill$\square$
\end{proof}

\noindent We observe that also for the sequential algorithm \eqref{eq-method-sequential} the convergence estimate for the feasibility violation depends explicitly on the minibatch size $N$ via the term $b_N^s$ (recall that  $b_N^s \to \infty$ as $N \to \infty$).  Since $b_N^s$ is an increasing sequence in $N$, it follows that  the larger is the minibatch size $N$ the better is also the complexity of the sequential algorithm~\eqref{eq-method-sequential} in terms of constraints feasibility.  In conclusion, for the sequential variant our rates prove that minibatching always helps and the feasibility  estimate  depends exponentially  on the minibatch size $N$.   On the other hand, the suboptimality estimate contains a term which does not depend on the minibatch size $N$ as it happens for feasibility violation  estimate.  Recall that for the parallel algorithm we proved that  minibatching works only for $L_N < 1$ and the estimates depend linearly on $L_N$. 

%%%%%%%%%%%%%%%%%%%%%%%%%%%%%

\section{Extensions}  
\label{sec:extensions}
\noindent In this section we discuss some possible extensions of the framework presented  in this paper related to the objective function, algorithms and stepsizes. Some of these extensions will be considered in our future work. \\
  
\noindent First, from our convergence analysis it is easy to note that the derivations still remain valid for a larger  class of objective functions  in  the model \eqref{eq-problem}.  More precisely,  we can replace  the boundedness on the subgradients of $f$ (Assumption~\ref{asum-base}(b)), i.e.   $ \|s_f(x)\| \le M_f$,  with a more general assumption, that is there exist two constants $M_{f,1}, M_{f,2} \geq 0$  such that the (sub)gradients of $f$ satisfy the following inequality:
 \[ \|s_f(x)\| \le M_{f,1} + M_{f,2} \| x - x^*\| \quad \forall  s_f(x) \in \partial f(x) \; \hbox{and} \;  x \in Y, x^* \in X^*.  \]
Clearly, this condition covers the class of functions with bounded subgradients, e.g.  take $M_{f,2}=0$, and also the class of functions with Lipschitz continuous gradients \cite{Nes:04}. Indeed, if there is $L_f>0$ such that the gradients $\nabla f$ satisfy: 
$$\|\nabla f(x) - \nabla f(y)\| \leq L_f \|x - y\| \quad  \forall x, y \in Y, $$ 
then we immediately get
$$ \|\nabla f(x)\| \leq \|\nabla f(x^*)\| +  \| \nabla f(x) - \nabla f(x^*)  \| \leq \|\nabla f(x^*)\| + L_f \| x - x^*  \|, $$ which proves our inequality for $M_{f,1} = \max_{x \in X^*} \|\nabla f(x^*)\|$ and $M_{f,2} = L_f$.  Our convergence analysis can be easily adapted  for this more general assumption, however, the recurrence relations will be more cumbersome.  For example, the recurrence from Lemma \ref{lem-viter} becomes now:
\begin{align*}
& \|v_{k+1} - x^*\|^2  +  2\a_k(1-\rho)\left(f(\Pi_X[x_k])- f^*\right)  \\ 
& \quad \le  \left(1 - \a_k \left(\rho\mu -2(1-\rho) M_{f,2} \right)  + 2 \a_k^2 M_{f,2}^2 \right) \|x_k - x^*\|^2 \\
&\qquad    + 2\a_k(1 -\rho) M_{f,1} \|\Pi_X[x_k] - x_k\|  + 2 \a_k^2 M_{f,1}^2.
\end{align*}

\noindent Second, when the objective function $f$ has an easy proximal operator  we can replace the  subgradient steps \eqref{eq-grad-update} and  \eqref{eq-vup} by a proximal point step:
\[   v_k  = \text{prox}_{\alpha_{k-1} f} (x_{k-1} )  = \arg \min_{y} f(y) + \frac{1}{2 \alpha_{k-1} } \|  y  -  x_{k-1}  \|^2.   \]
An algorithm combining the proximal point step with a single feasibility step (i.e., $N=1$) has been considered in  \cite{PatNec:17} and convergence rates of order $\mathcal{O}(1/t)$   have been  proved provided that the objective function is smooth (i.e., it has  Lipschitz continuous  gradient)  and strongly convex. Note that it is easy to extend that convergence analysis to the minibatch settings following the framework developed in this paper. \\

\noindent Third extension is still related to the objective function,  by considering $f$ in the  composite form, i.e.:
\[  f(x) = f_{1}(x) + f_{2}(x),  \]
where $f_1$ is smooth and $f_2$ can be non-smooth but  admits an easy proximal operator. Note that  if the set  $Y$ is present in the optimization model \eqref{eq-problem}, then it can be included in $f_2$ as the indicator function.  For this composite objective function, steps \eqref{eq-grad-update} and  \eqref{eq-vup} can be replaced by:
\begin{align*}
 v_k & = \text{prox}_{\alpha_{k-1} f_{2}} (x_{k-1}  - \alpha_{k-1} \nabla f_{1}(x_{k-1})) \\
   & = \arg \min_{y} f_2(y) + \frac{1}{2 \alpha_{k-1} } \|  y  - ( x_{k-1}  - \alpha_{k-1} \nabla f_{1}(x_{k-1})) \|^2. 
\end{align*} 
Note that for  $ f_{2}(x) = 1_{Y}(x)$, the indicator function of the convex set  $Y$, we recover the updates \eqref{eq-grad-update} and  \eqref{eq-vup}. Hence, it will be interesting to extend our  convergence analysis to this general composite  objective function $f$.\\   

\noindent Finally, in the parallel algorithm (see \eqref{eq-method}) the feasibility steps depend on an extrapolated  stepsize $\beta \in (0, \, 2/L_N)$.  When $L_N$ cannot be computed explicitly, we propose to approximate it online with $L_N^k$, and use  at each iteration an adaptive extrapolated  stepsize  $\beta_k$ of the form:
\begin{align}
\label{eq:betak}
\beta_k =  \frac{2-\delta}{L_N^k}   =       \frac{2-\delta}{N} \sum_{i=1}^N\frac{(g_{\o_{k}^{i}}^+(v_k))^2}{\|d_k^i\|^2} \Big{/} \left\| \frac{1}{N}\sum_{j=1}^N\frac{g_{\o_{k}^{j}}^+(v_k)}{\|d_k^j\|^2}\, d_k^j\right\|^2,
\end{align}
for some $\delta \in (0, \;  2)$ sufficiently  small. The convergence rate of the parallel algorithm  for this adaptive choice \eqref{eq:betak} of the stepsize $\beta_k$ will be analyzed in our future work (see e.g.,  \cite{NecNed:19} for some preliminary results related to the convex feasibility problem).

%%%%%%%%%%%%%%%%%%%%%%%%%%%%%%5

\section{Preliminary numerical results} 
\label{sec:sim}
\noindent Many data-driven optimization applications can be formulated as  convex optimization problems with the objective function composed of a quadratic term and a regularizer and constraints (so-called \textit{constrained Lasso}) of the form:
\[  \min_{x \in [l, \, u]}  \|Hx -y\|^2 +  \lambda \|Dx\|_1   \quad \text{s.t.}   \quad Ax+b \leq 0,  \]
where the problem is parametrised by the data (measurements) $y$, $H$ is an appropriate linear operator (e.g.,  the forward operator,  the circular convolution)   and $D$ is another linear operator (e.g., the identity, the finite difference or the  Wavelet transform). Additionally, we impose constraints of the form $x \in X=[l, \; u] \cap \{ x:  \; Ax + b \leq 0 \}$, where $A \in \R^{m \times n}$. The constrained Lasso arises e.g., in image deblurring or denoising, computerised tomography or some  inverse problems, see.e.g  \cite{BotHei:12}.  Note that for this formulation the strong convexity Assumption~\ref{asum-base}(b) holds for full column  matrices $H$ (see e.g., \cite{NecNes:15}) and also the  linear regularity Assumption~\ref{asum-regularmod}  holds (see e.g.  \cite{NecRic:18}). Moreover, the set $Y=[l,\, u]$ is compact so that the objective function has bounded subgradients and the functional constraints $g_\o (x) = a_\o^T x + b_\o$ are linear and, consequently, Assumption~\ref{asum-base} also holds.

\noindent In our experiments we use synthetic data, where $H$ is Toeplitz-like matrix and $D$ the finite difference operator (as in image deblurring \cite{BotHei:12}). We also generate $A$ randomly, with $m=3n$ constraints. We consider a partition of $\A =\{1, 2, \cdots, m\} $ of equal size $N$, i.e., $\A=\cup_{i=1}^\ell J_i$. Hence, $m= N \cdot \ell$. We compute $L_N$ as in \eqref{eq:LNlin} for this partition. We consider full iterations, i.e. we plot the behavior of the algorithms over epochs  $t N/m$ (number of passes over all the rows of matrix $A$).

\noindent In the first set of experiments we compare the parallel  (see \eqref{eq-method}) and sequential (see \eqref{eq-method-sequential}) algorithms for different minibatch sizes $N=1, 50$ and $100$ on a constrained Lasso problem with $n=10^3$.  The plots in Fig. 1 present the convergence behavior of these algorithms in terms of  feasibility violation of the average point over full iterations $t N/m$: parallel algorithm (left) and sequential algorithm (right). As we can see from Fig. 1, increasing the minibatch size $N$ usually leads to better convergence for both algorithms. 
\begin{figure}
\label{fig_lcth} 
\includegraphics[height=5cm,width=6.2cm]{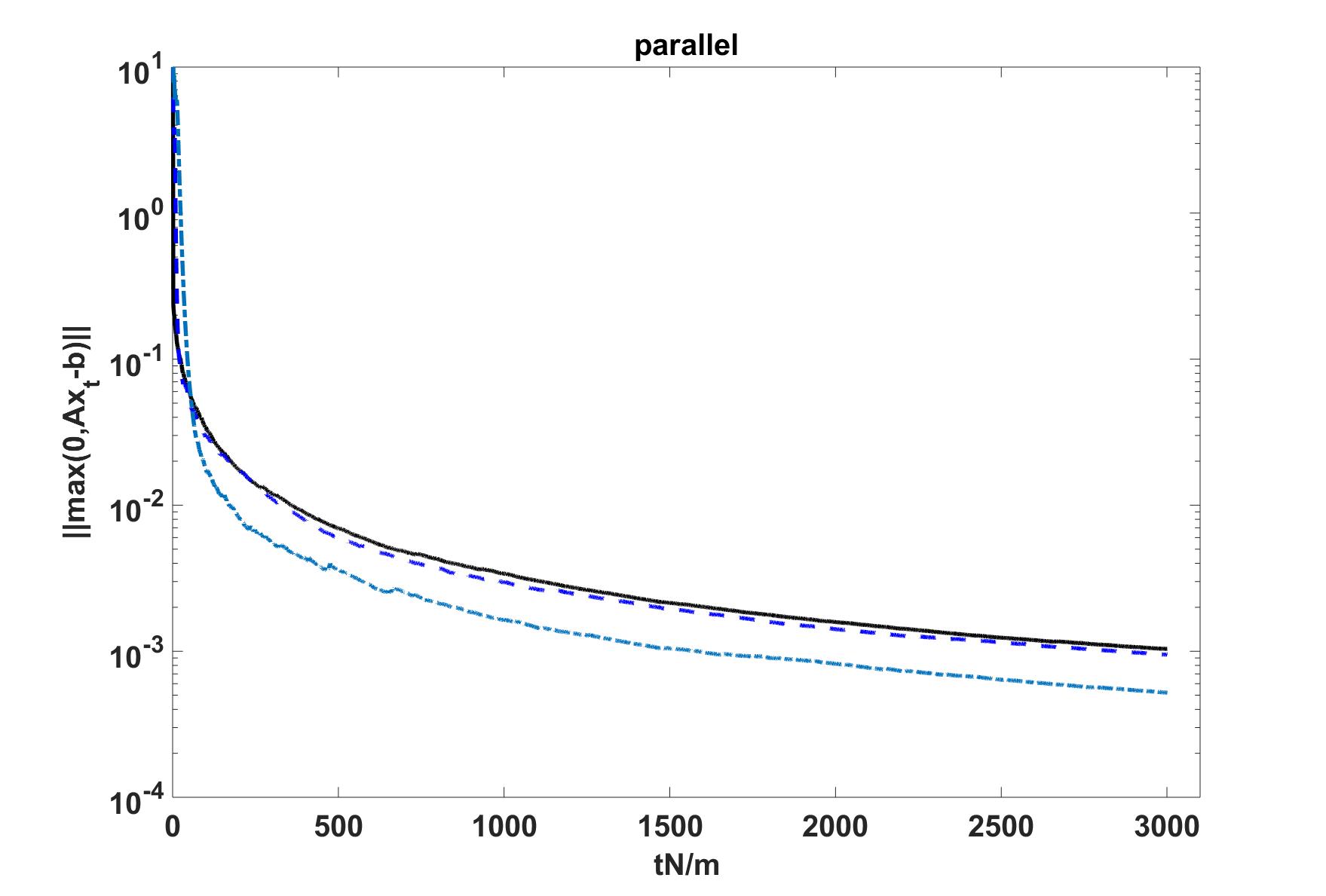}
\hspace*{-0.5cm}
\includegraphics[height=5cm,width=6.2cm]{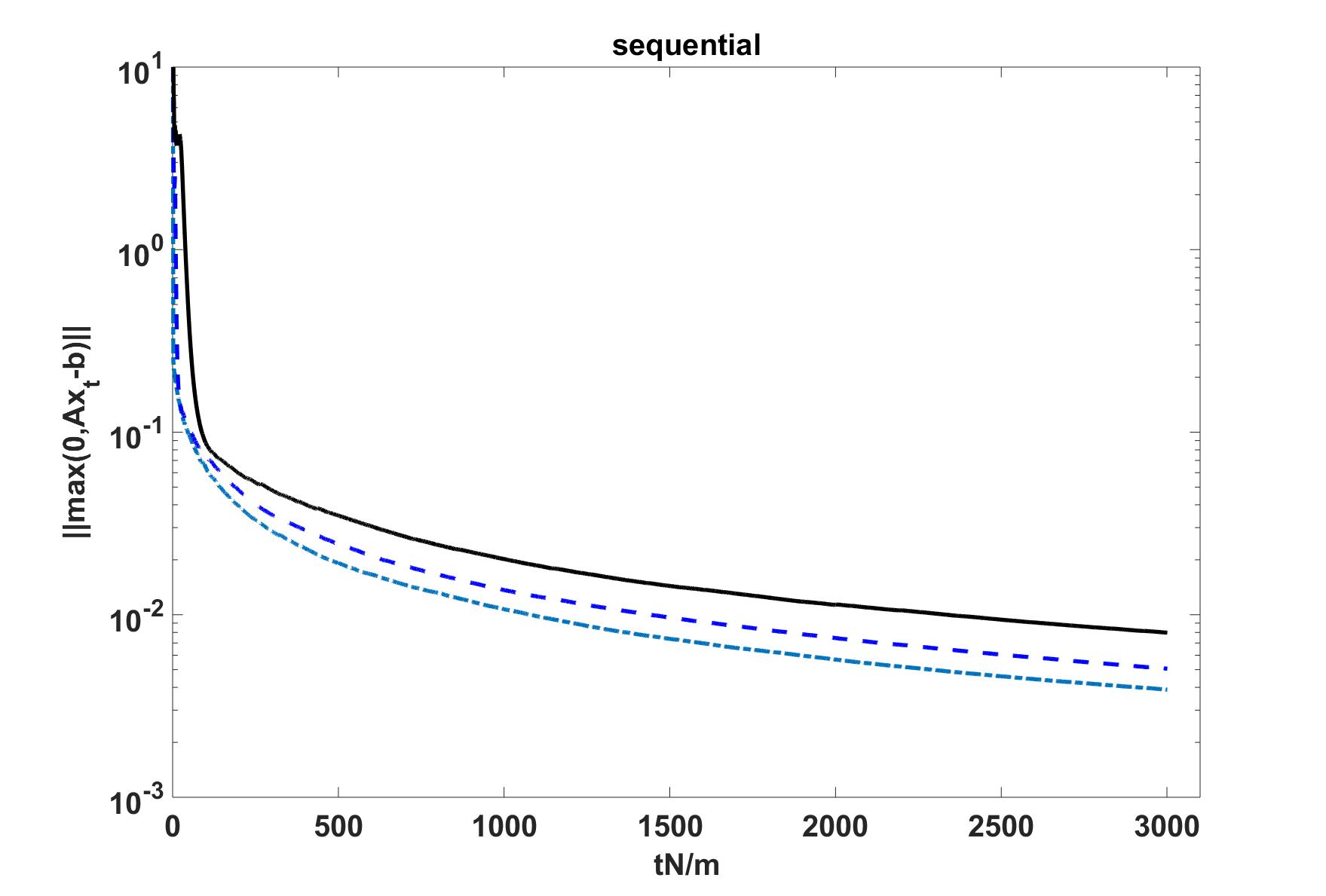}
\caption{Convergence  of  minibatch  parallel (left)  and sequential (right)  algorithms for different minibatch sizes: $N=1$ (solid), $N=50$ (dashed) and $N=100$ (dot-dashed). }
\end{figure}

\noindent Then, we   compare  the parallel algorithm  with the extrapolated stepsize $\b=1.9/L_N$ and  the sequential algorithm  with $\b=1.9$. The results on a problem of dimension $n=10^3$ and minibatch size  $N=10$ are displayed in Fig. 2:  suboptimality (left) and feasibility violation (right) in the average point over full iterations. We observe a faster convergence for the sequential algorithm, as our theory also predicted. 
\begin{figure}
\label{fig_lcth} 
\includegraphics[height=5cm,width=6.2cm]{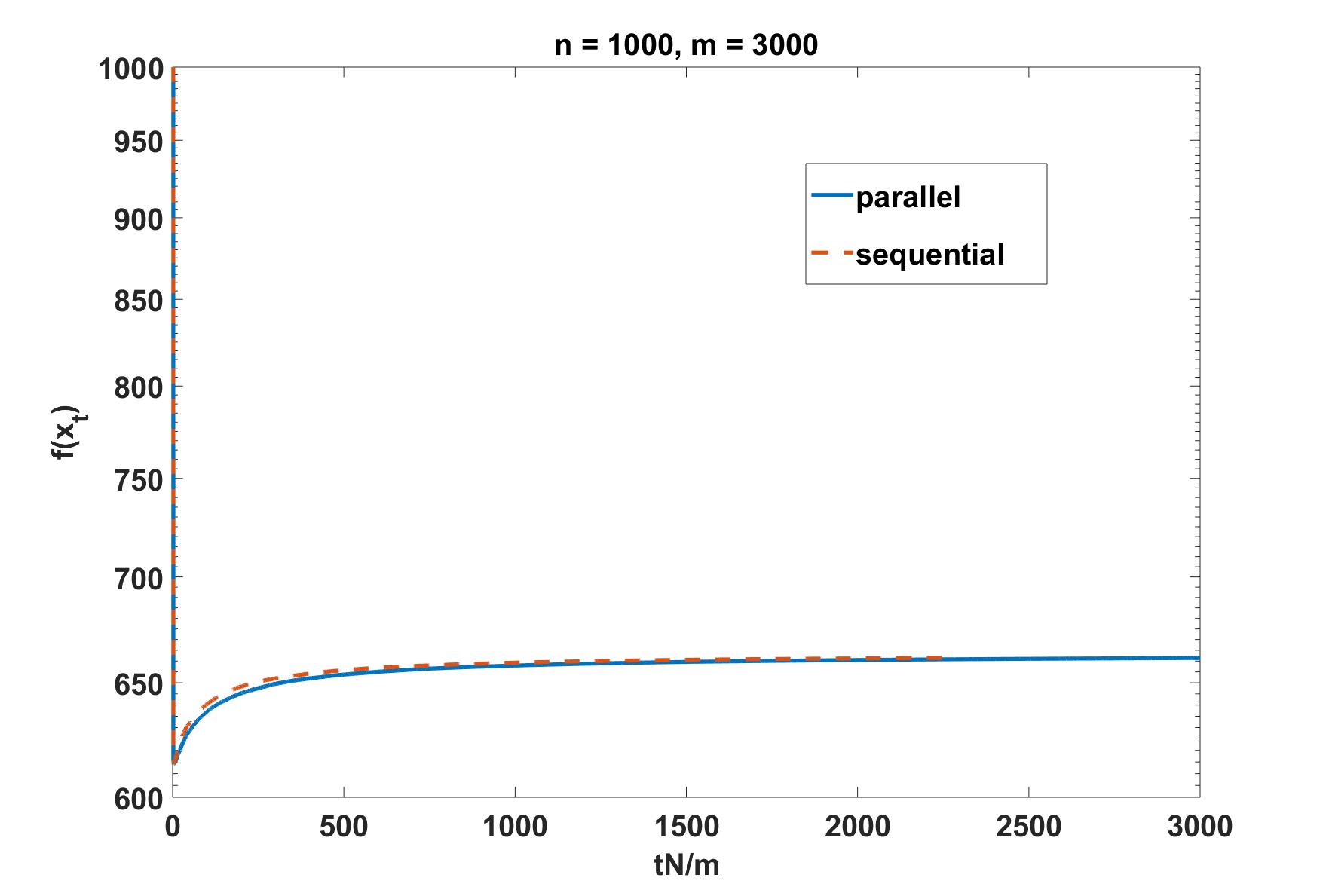}
\hspace*{-0.5cm}
\includegraphics[height=5cm,width=6.2cm]{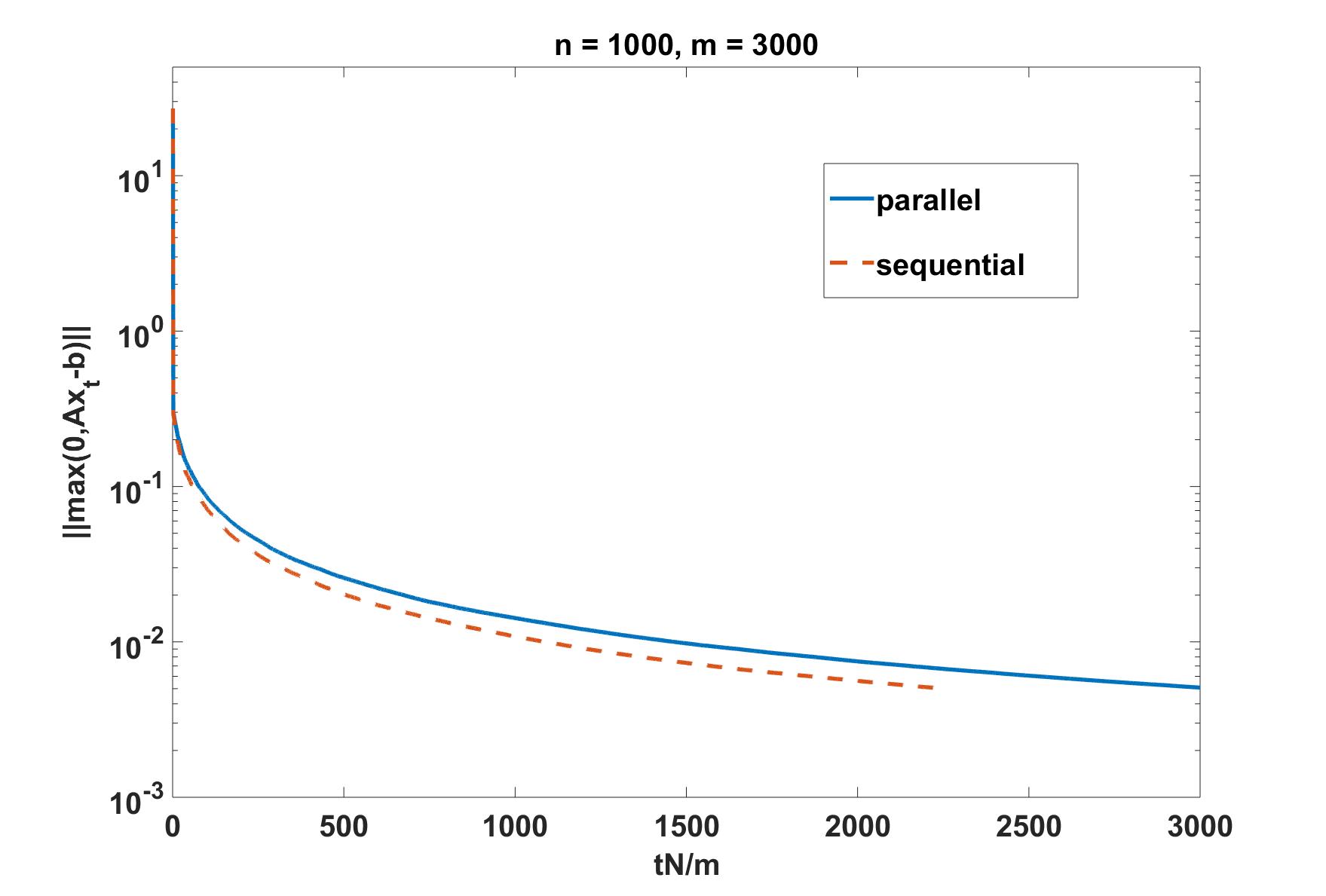}
\caption{Convergence behavior  of the minibatch  parallel (solid) and sequential (dashed) algorithms: objective function (left)  and  feasibility violation (right). }
\end{figure}

\noindent Finally, we   compare  the parallel algorithm \eqref{eq-method} based on our extrapolated  stepsize $\b=1.9/L_N$ and  a variant with fixed stepsize $\b=1.9$. The results on a constrained Lasso problem of dimension $n=10^3$ and minibatch size  $N=10$ are displayed in Fig. 3:  suboptimality (left) and feasibility violation (right). We observe that extrapolation  $\b = 1.9/L_N > 2$ accelerates  substantially the parallel algorithm in terms of feasibility criterion.  Note also that all the plots  show a $\mathcal{O}(1/t)$ rate for the average sequence in the feasibility criterion, thus supporting our theoretical findings. 
\begin{figure}
\label{fig_lcth} 
\includegraphics[height=5cm,width=6.2cm]{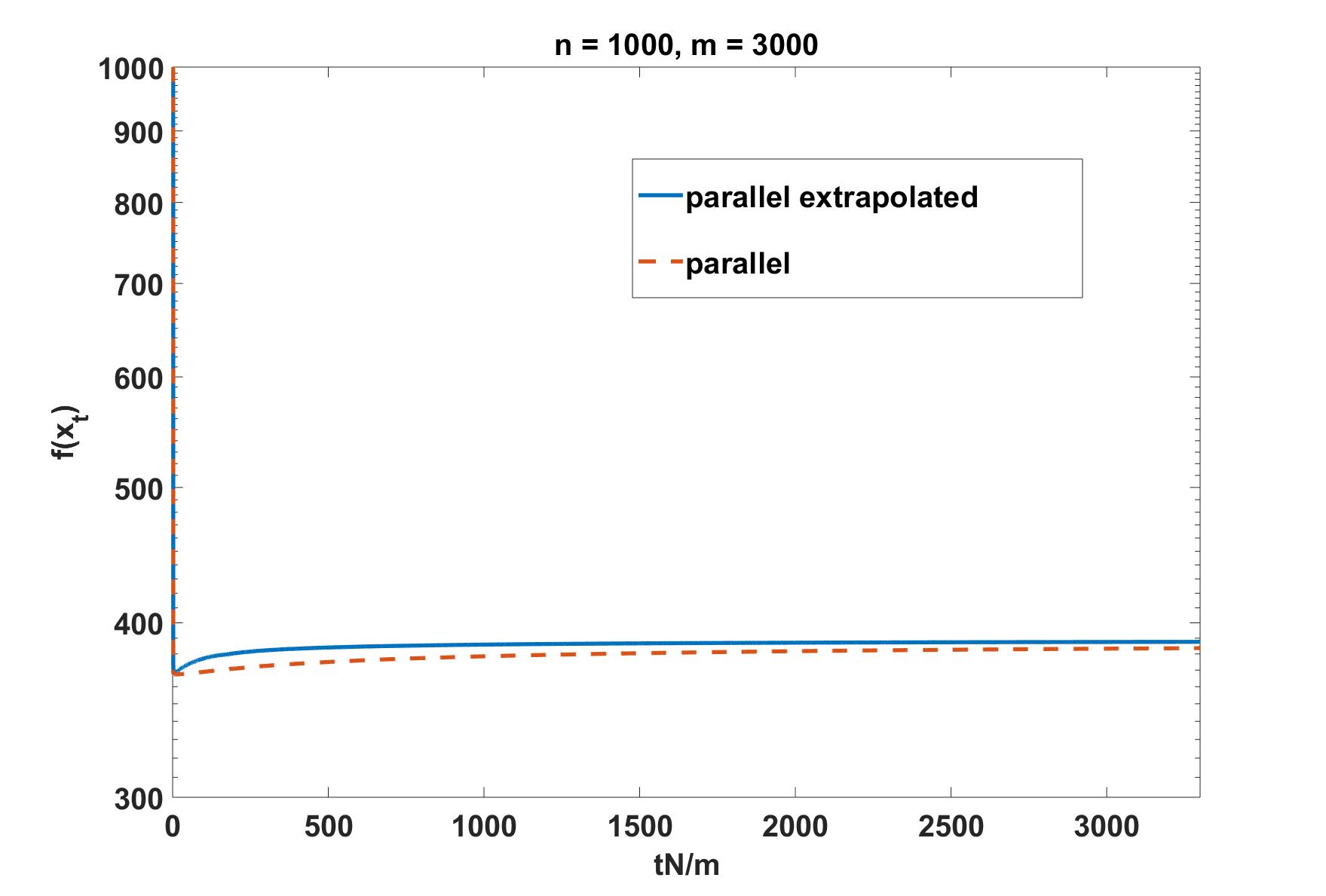}
\hspace*{-0.5cm}
\includegraphics[height=5cm,width=6.2cm]{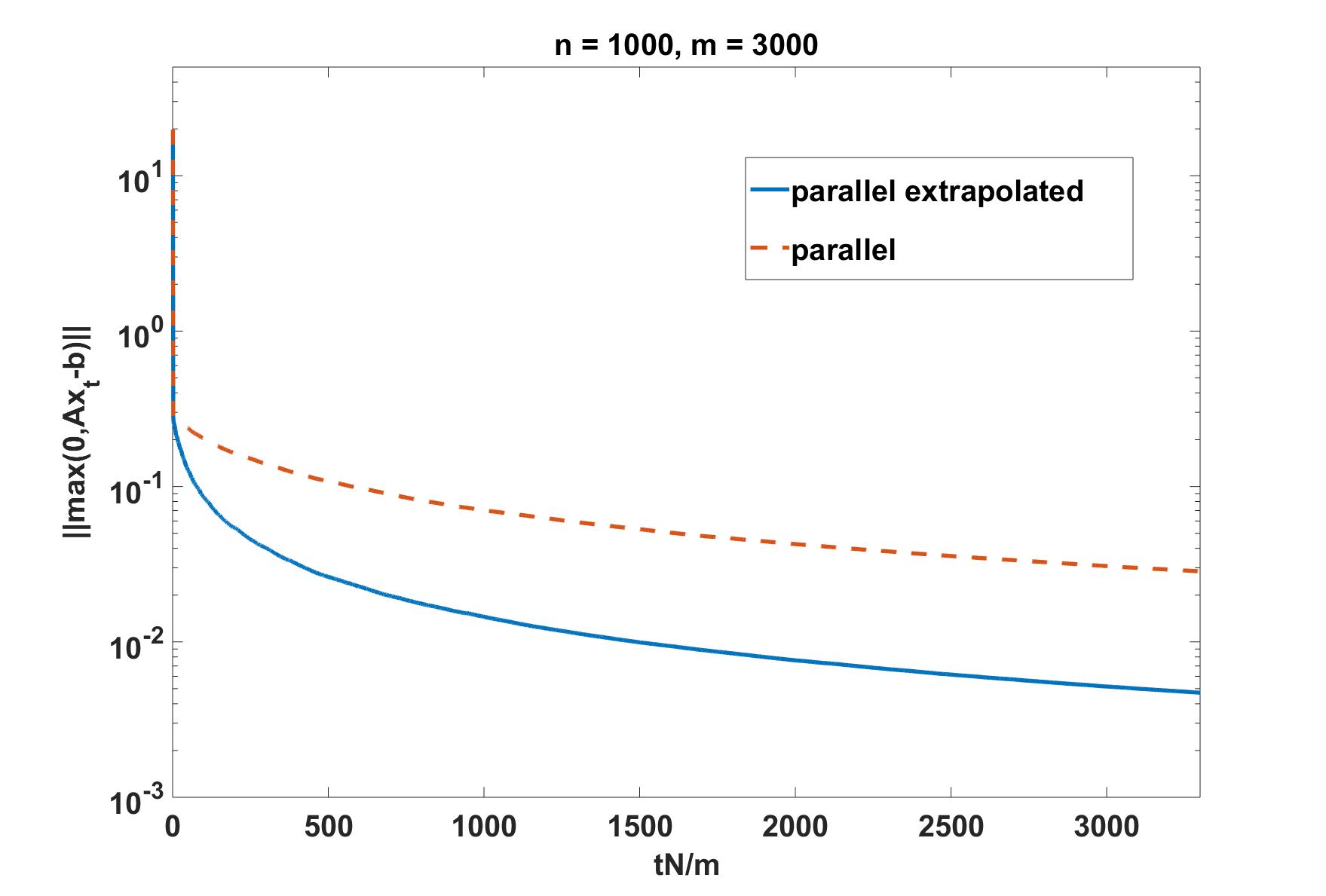}
\caption{Behavior  of  parallel  algorithm with extrapolated stepsize  $\b=1.9/L_N$ (solid) and  fixed stepsize $\b=1.9$ (dashed): objective function (left), feasibility violation (right). }
\end{figure}

%%%%%%%%%%%%%%%%%%%%%%%%%%%%

\section{Conclusions}
In this paper we have considered  (non-smooth) convex optimization problems with (possibly) infinite intersection of constraints. For solving this general class of convex  problems  we  have proposed   subgradient  algorithms with random minibatch feasibility steps. At each iteration, our algorithms take first  a step for minimizing the objective function and then a subsequent step   minimizing the feasibility violation of the observed minibatch of constraints. The feasibility updates were performed based on either  parallel or sequential random observations of several  constraint components. For a diminishing stepsize and for  strongly convex objective functions, we have proved sublinear convergence rates   for the expected distances of the weighted averages of the iterates from the constraint set, as well as for the expected suboptimality of the function values along the weighted averages.  Our convergence  rates  are  optimal for  subgradient methods  with random feasibility steps for solving this class of non-smooth convex problems. Moreover,  the rates  depend explicitly on the minibatch size. From our knowledge, this work is the first deriving conditions when minibatching works for subgradient methods with random minibatch  feasibility updates and proving how better  is their complexity compared to  the non-minibatch  variants. Finally, our convergence analysis shows that for  the sequential algorithm minibatching always helps and the feasibility  estimate  depends exponentially  on the minibatch size, while  for the parallel algorithm we proved that minibatching works only when some parameter of the optimization problem is strictly less than 1. The numerical results also support the convergence results.

%%%%%%%%%%%%%%%%%%%%%%%%%%%%%%%%%%
%%%%%%%%%%%%%%%%%%%%%%%%%%%%%%%%%%

\begin{acknowledgements}
This research was supported by the National Science Foundation under CAREER grant CMMI 07-42538 and  by the Executive Agency for Higher Education, Research and Innovation Funding (UEFISCDI), Romania,  PNIII-P4-PCE-2016-0731, project ScaleFreeNet, no. 39/2017.
\end{acknowledgements}

%%%%%%%%%%%%%%%%%%%%%%%%

\end{document}